\documentclass{amsart}

\usepackage{amssymb}
\usepackage{mathrsfs}
\usepackage{stmaryrd}
\usepackage{imakeidx}
\usepackage[colorlinks = true]{hyperref}
\usepackage{enumitem}
\usepackage{comment}
\input xy \xyoption {all}
\usepackage{bbm}

\usepackage{tikz}
\usetikzlibrary{cd}

\usepackage{dsfont}

\usepackage{rotating}
\usepackage{lscape}

\newcommand{\Z}{\mathbb{Z}}

\newcommand{\F}{\mathbb{F}}

\newcommand{\bk}{\Bbbk}
\newcommand{\Ga}{\mathbb{G}_{\mathrm{a}}}
\newcommand{\Gm}{\mathbb{G}_{\mathrm{m}}}

\newcommand{\Mod}{\mathsf{Mod}}

\newcommand{\Kb}{K^{\mathrm{b}}}
\newcommand{\mix}{{\mathrm{mix}}}

\newcommand{\Db}{D^{\mathrm{b}}}
\newcommand{\Dmix}{\mathsf{D}^\mix}
\newcommand{\Pmix}{\mathsf{P}^\mix}

\newcommand{\Perv}{\mathsf{Perv}}

\newcommand{\hatstar}{\mathbin{\widehat{\star}}}

\newcommand{\For}{\mathsf{For}}
\newcommand{\IC}{\mathscr{I}\hspace{-1pt}\mathscr{C}}

\newcommand{\id}{\mathrm{id}}

\newcommand{\simto}{\xrightarrow{\sim}}

\DeclareMathOperator{\Hom}{Hom}

\newcommand{\Iw}{\mathrm{I}}
\newcommand{\Iwu}{\mathrm{I}_{\mathrm{u}}}
\newcommand{\Fl}{\mathrm{Fl}}
\newcommand{\tFl}{\widetilde{\mathrm{Fl}}}
\newcommand{\Gr}{\mathrm{Gr}}
\newcommand{\Wf}{W_{\mathrm{f}}}

\newcommand{\Sf}{S_{\mathrm{f}}}

\newcommand{\bG}{\mathbf{G}}

\newcommand{\bB}{\mathbf{B}}

\newcommand{\bT}{\mathbf{T}}

\newcommand{\bW}{\mathbf{W}}
\newcommand{\bWf}{\mathbf{W}_{\mathrm{f}}}
\newcommand{\bS}{\mathbf{S}}
\newcommand{\bSf}{\mathbf{S}_{\mathrm{f}}}

\newcommand{\bWaff}{\bW_{\mathrm{aff}}}
\newcommand{\bSaff}{\mathbf{S}_{\mathrm{aff}}}

\newcommand{\bg}{\mathbf{g}}
\newcommand{\bb}{\mathbf{b}}
\newcommand{\bt}{\mathbf{t}}

\newcommand{\Loop}{\mathrm{L}}

\makeatletter
\def\lotimes{\@ifnextchar_{\@lotimessub}{\@lotimesnosub}}
\def\@lotimessub_#1{\mathchoice{\mathbin{\mathop{\otimes}^L}_{#1}}%
  {\otimes^L_{#1}}{\otimes^L_{#1}}{\otimes^L_{#1}}}
\def\@lotimesnosub{\mathbin{\mathop{\otimes}^L}}
\makeatother

\makeatletter
\def\lboxtimes{\@ifnextchar_{\@lboxtimessub}{\@lboxtimesnosub}}
\def\@lboxtimessub_#1{\mathchoice{\mathbin{\mathop{\boxtimes}^L}_{#1}}%
  {\boxtimes^L_{#1}}{\boxtimes^L_{#1}}{\boxtimes^L_{#1}}}
\def\@lboxtimesnosub{\mathbin{\mathop{\boxtimes}^L}}
\makeatother

\newcommand{\scO}{\mathscr{O}}

\newcommand{\cK}{\mathcal{K}}
\newcommand{\cJ}{\mathcal{J}}

\newcommand{\Rep}{\mathsf{Rep}}

\newcommand{\st}{\mathsf{t}}

\newcommand{\sfD}{\mathsf{D}}

\newcommand{\sfK}{\mathsf{K}}
\newcommand{\sfP}{\mathsf{P}}
\newcommand{\sfT}{\mathsf{T}}

\newcommand{\scB}{\mathscr{B}}

\newcommand{\scE}{\mathscr{E}}
\newcommand{\scF}{\mathscr{F}}
\newcommand{\scG}{\mathscr{G}}

\newcommand{\scM}{\mathscr{M}}

\newcommand{\scT}{\mathscr{T}}

\newcommand{\can}{\mathrm{can}}

\newcommand{\BSRep}{\mathsf{BSRep}}
\newcommand{\SRep}{\mathsf{SRep}}

\newcommand{\Waff}{W_{\mathrm{aff}}}
\newcommand{\Saff}{S_{\mathrm{aff}}}

\newcommand{\Spec}{\mathrm{Spec}}
\newcommand{\bbJ}{\mathbb{J}}
\newcommand{\bbI}{\mathbb{I}}

\newcommand{\Par}{\mathsf{Par}}
\newcommand{\ParBS}{\mathsf{Par}^{\mathrm{BS}}}
\newcommand{\ph}{{}^p \hspace{-1pt} h}
\newcommand{\sfTBS}{\mathsf{T}^{\mathrm{BS}}}

\newcommand{\add}{\mathrm{add}}
\newcommand{\BSK}{\mathsf{BSK}}
\newcommand{\BSKr}{\mathsf{BSK}_{\mathrm{r}}}
\newcommand{\sv}{\mathsf{v}}
\newcommand{\rmX}{\mathrm{X}}

\newcommand{\fR}{\mathfrak{R}}

\numberwithin{equation}{section}
\numberwithin{figure}{section}
\newtheorem{thm}{Theorem}[section]
\newtheorem{lem}[thm]{Lemma}
\newtheorem{prop}[thm]{Proposition}
\newtheorem{cor}[thm]{Corollary}

\theoremstyle{definition}

\theoremstyle{remark}
\newtheorem{rmk}[thm]{Remark}
\newtheorem{ex}[thm]{Example}

\title[Mixed modular perverse sheaves on affine flag varieties]{Mixed modular perverse sheaves on affine flag varieties and Koszul duality}

 \author[S.~Riche]{Simon Riche}
 \address{Universit\'e Clermont Auvergne, CNRS, LMBP, F-63000 Clermont-Ferrand, France.}
 \email{simon.riche@uca.fr}

\thanks{This project has received funding from the European Research Council (ERC) under the European Union's Horizon 2020 research and innovation programme (grant agreement No.~101002592).}


\begin{document}

\begin{abstract}
Under some technical assumptions, and building on joint work with Bezrukavnikov,
we prove a multiplicity formula for indecomposable tilting perverse sheaves on affine flag varieties, with coefficients in a field of characteristic $p$, in terms of $p$-Kazhdan--Lusztig polynomials.
Under the same assumptions,
we also explain the construction of a ``degrading functor'' relating mixed modular perverse sheaves (as defined in joint work with Achar) on such varieties to ordinary perverse sheaves.
\end{abstract}

\maketitle

\section{Introduction}
\label{sec:intro}

\subsection{Mixed perverse sheaves}

Many applications of perverse sheaves to problems in Representation Theory require in one form or another a use of \emph{mixed} perverse sheaves in the sense of Deligne (see e.g.~\cite[\S 5.1]{bbd}). This has long been an obstacle to the application of such methods for problems involving fields of positive characteristic since, whereas the definition of perverse sheaves with coefficients in any field causes no problem, the translation of the definition of their \emph{mixed} counterparts in this setting seems hopeless.\footnote{The main reason for that is that Deligne's definition is stated in terms of properties of eigenvalues of the Frobenius, which should involve some powers of the cardinality $q$ of the field of definition of the varieties under consideration; now powers of an integer in a field of characteristic $0$ or of positive characteristic behave in a drastically different way!}

As a way to bypass this difficulty, we proposed in joint work with Achar~\cite{modrap2} a definition of mixed modular\footnote{In this context, ``modular'' always means ``with coefficients in a field of positive characteristic.''} perverse sheaves on some varieties (including flag varieties of Kac--Moody groups). This definition might seem artificial, but it has proven useful in several constructions involving categories of representations of reductive algebraic groups over fields of positive characteristic, see e.g.~\cite{prinblock,amrw}. Our approach was mainly suggested by works of Be{\u\i}linson--Ginzburg--Soergel~\cite{bgs} for $\ell$-adic sheaves, and the recent (at that point) theory of parity complexes developed by Juteau--Mautner--Williamson~\cite{jmw}. Since then more involved approaches to this question have been proposed, in particular by Eberhardt--Scholbach~\cite{es} based on the use of motivic sheaves; they give rise to the same category, but now enhanced with more operations (i.e.~a 6-functors formalism), and is applicable to more general contexts. 

\subsection{The search for a forgetful functor}

Despite this progress, a central question remains unanswered: what is the relation between mixed perverse sheaves and ordinary perverse sheaves? Namely, in the $\ell$-adic setting, mixed perverse sheaves are complexes of sheaves on a variety $X_\circ$ defined over a finite field $\mathbb{F}_q$, and we have a canonical functor to ``ordinary'' perverse sheaves, i.e.~perverse sheaves on the fiber product $X=X_\circ \otimes_{\mathbb{F}_q} \overline{\mathbb{F}_q}$ (where $\overline{\mathbb{F}_q}$ is an algebraic closure of $\mathbb{F}_q$), given by pullback along the projection $X \to X_\circ$. Informally, this functor ``forgets the mixed structure;'' it has various nice properties, it particular it sends simple objects to simple objects. If one modifies the category of mixed perverse sheaves slightly following~\cite{bgs}, this functor is even a ``degrading functor,'' i.e.~it behaves like the forgetful functor from graded modules over a finite-dimensional algebra to ordinary modules over that algebra. This functor is very useful; for instance it is at the heart of the proof of the Decomposition Theorem in~\cite{bbd}, and appears in one form or another in all the proofs that the dimensions of fibers of intersection cohomology complexes on flag varieties are computed by Kazhdan--Lusztig polynomials (see~\cite{kl}).

In the modular setting, one directly works with complexes on $X$ (either complexes of parity complexes in the approach of~\cite{modrap2}, or some motivic sheaves in the approach of~\cite{es}), and there exists a priori no easy way to ``forget the mixed structure'' and recover an ordinary perverse sheaf. This is a pitty because, even if the theory is way more incomplete than in the $\ell$-adic setting, there are some computations that one can do with mixed modular perverse sheaves (e.g.~describe multiplicities in indecomposable tilting perverse sheaves over flag varieties, see~\cite{amrw}) and for which no counterpart for ordinary perverse sheaves exists.

This problem was solved in~\cite{modrap2} in the (important) case when $X$ is the flag variety of a connected reductive group, and the characteristic of the field of coefficients is good. The main results of this paper are an extension of this construction to the case of \emph{affine} flag varieties, and an application to a formula for multiplicities of ordinary indecomposable tilting perverse sheaves on such varieties. (In the case of $\ell$-adic sheaves, a similar formula is due to Yun~\cite{yun}.)

\subsection{Main results}

Let $\F$ be an algebraically closed field of positive characteristic, and let $\bk$ be an algebraic closure of a finite field whose characteristic $p$ is invertible in $\F$. Let also $G$ be a connected reductive algebraic group over $\F$, and choose a Borel subgroup $B \subset G$ and a maximal torus $T \subset B$. Then one can consider the loop group $\Loop G$ attached to $G$, the Iwahori subgroup $\Iw$ determined by $B$, its pro-unipotent radical $\Iwu$, and the affine flag variety $\Fl_G = \Loop G / \Iw$. Consider also the category $\Perv_{\Iwu}(\Fl_G,\bk)$ of $\Iwu$-equivariant $\bk$-perverse sheaves on $\Fl_G$. This category has several collections of important objects, whose classes all form bases of the Grothendieck group:
\begin{enumerate}
 \item 
 The simple objects are the intersection cohomology complexes $(\IC_w : w \in W)$ attached to $\Iwu$-orbits on $\Fl_G$; they are naturally parametrized by the extended affine Weyl group $W = \Wf \ltimes \rmX_*(T)$ where $\Wf$ is the Weyl group of $(G,T)$.
 \item 
 For any $w \in W$, we have ``standard'' and ``costandard'' perverse sheaves $\Delta_w$ and $\nabla_w$, obtained by $!$-extension, resp.~$*$-extension, of the constant perverse sheaf on the orbit labelled by $w$.
 \item
 For any $w \in W$, we also have an indecomposable tilting perverse sheaf $\scT_w$ whose support is the closure of the orbit of $w$.
\end{enumerate}
For various reasons, it is a very interesting problem for Geometric Representation Theory to describe the combinatorics of these objects, and in particular to describe the coefficients of the expansion of classes of simple objects or indecomposable tilting objects in the basis of the Grothendieck group given by classes of standard perverse sheaves.\footnote{It is a basic fact that, for any $w \in W$, the classes of $\Delta_w$ and $\nabla_w$ in the Grothendieck group coincide, so that there is no need to choose between the two families.} For any indecomposable tilting perverse sheaf $\scT_w$, the coefficient of the class of $\Delta_y$ in this expansion will be denoted $(\scT_w : \Delta_y)$.

On the other hand, following~\cite{modrap2} one can consider the ``mixed modular derived category'' $\Dmix_{\Iwu}(\Fl_G,\bk)$, defined as the bounded homotopy category of the category of $\Iwu$-equivariant parity complexes on $\Fl_G$ in the sense of~\cite{jmw}. This category admits a canonical t-structure which we call the ``perverse t-structure,'' and whose heart is denoted $\Perv^{\mix}_{\Iwu}(\Fl_G,\bk)$. This latter category looks very much like ``a graded version of $\Perv_{\Iwu}(\Fl_G,\bk)$;'' in particular we have families of objects as above:
\begin{enumerate}
 \item 
 The simple objects (up to ``Tate twist'') are some ``mixed intersection cohomology complexes'' $(\IC^{\mix}_w : w \in W)$ attached to $\Iwu$-orbits on $\Fl_G$; they are naturally parame\-trized by $W$.
 \item 
 For any $w \in W$, we have ``standard'' and ``costandard'' mixed perverse sheaves $\Delta^\mix_w$ and $\nabla^\mix_w$.
 \item
 For any $w \in W$, we also have an indecomposable mixed tilting perverse sheaf $\scT^\mix_w$ whose support is the closure of the orbit of $w$.
\end{enumerate}
As above, classes of standard mixed perverse sheaves\footnote{In this setting it is no longer true that the classes of $\Delta^\mix_w$ and $\nabla^\mix_w$ coincide, but they are related by a simple operation, so that again the corresponding bases play essentially the same roles.} form a basis of the Grothen\-dieck group (now over the ring $\Z[v,v^{-1}]$, where $v$ corresponds to Tate twist), and one can ask what is the expansion of classes of simple or indecomposable tilting perverse sheaves in this basis. This question is largely open for simple objects, but for indecomposable tilting objects an answer is given (assuming that $p$ is odd and very good) in~\cite{amrw}: these coefficients are the $p$-Kazhdan--Lusztig polynomials $(\ph_{y,w} : y,w \in W)$ of the corresponding Hecke algebra in the sense of~\cite{jw}.

\begin{figure}
 \begin{tabular}{|c|c|c|c|c|c|c|c|c|}
  \hline
  $\mathbf{A}_n$ ($n \geq 1$) & $\mathbf{B}_n$ ($n \geq 2$) & $\mathbf{C}_n$ ($n \geq 3$) & $\mathbf{D}_n$ ($n \geq 4$) & $\mathbf{E}_6$ & $\mathbf{E}_7$ & $\mathbf{E}_8$ & $\mathbf{F}_4$ & $\mathbf{G}_2$ \\
  \hline
  $1$ & $n$ & $2$ & $2$ & $3$ & $19$ & $31$ & $3$ & $3$\\
  \hline
 \end{tabular}
\caption{Bounds on $p$}
\label{fig:bounds}
\end{figure}


Now we assume that $p$ is very good for $G$, and that moreover for any indecomposable summand of the root system of $(G,T)$, $p$ is strictly larger than the corresponding bound in Figure~\ref{fig:bounds}. Under these assumptions, the main results of the paper are the following.
\begin{enumerate}
 \item 
 (Theorem~\ref{thm:koszul-duality})
 The construction of an exact degrading (with respect to the Tate twist) functor
 \[
  \nu : \Perv^\mix_{\Iwu}(\Fl_G,\bk) \to \Perv_{\Iwu}(\Fl_G,\bk)
 \]
which satisfies
\[
 \nu(\IC_w^\mix) \cong \IC_w, \quad \nu(\Delta_w^\mix) \cong \Delta_w, \quad \nu(\nabla_w^\mix) \cong \nabla_w, \quad \nu(\scT_w^\mix) \cong \scT_w.
\]
\item
\label{it:intro-main}
(Corollary~\ref{cor:mult-tilt})
A proof that for any $y,w \in W$ we have
\[
(\scT_w : \Delta_y) = \ph_{y,w}(1).
\]
\end{enumerate}

\begin{rmk}
\begin{enumerate}
\item
In the body of the paper we work under slightly weaker assumptions, which however require more notation. These assumptions allow e.g.~the group $G=\mathrm{GL}_n(\F)$ and any prime $p$ invertible in $\F$.
\item
We expect that these assumptions can be weakened at least to ``$p$ is good for $G$.''
\item
We emphasize that in~\eqref{it:intro-main} we consider multiplicities in \emph{ordinary} (not mixed!) indecomposable tilting perverse sheaves. We do not know any more direct way of computing these integers.
\end{enumerate}
\end{rmk}

\subsection{Comments on the proof and further results}

As in the case of ``finite'' flag varieties in~\cite{modrap2},
these two results are deduced from the construction of a ``degrading'' functor relating $\Iwu$-equivariant parity complexes and tilting perverse sheaves on $\Fl_G$. In~\cite{modrap2} this construction relied on the results of the companion paper~\cite{modrap1}. Here it is deduced from the main result of~\cite{reg-quotient-pt2}, using relatively elementary manipulations with some ``Hecke categories'' whose definition is inspired by some constructions of Abe. (See~\S\S\ref{ss:Hecke-a-la-Abe}--\ref{ss:graded-Hecke} for the definition of these categories.)

As in~\cite{modrap2} our construction is \emph{non canonical}: it requires a choice of a ``pseudo-logarithm'' morphism (see~\S\ref{ss:relation-mult-add}) for the Langlands dual group $G^\vee_\bk$. This choice is necessary to relate the monodromy that naturally appears in the study of tilting perverse sheaves, and which involves the algebra of functions on the maximal torus $T^\vee_\bk$ of $G^\vee_\bk$, and the $\Iw$-equivariant cohomology of a point that naturally appears in the study of parity complexes, and which involves the algebra of functions on the dual of the Lie algebra of $T^\vee_\bk$. The interplay between these ``multiplicative'' and ``additive'' algebras seems to be a common feature of the various approaches to ``mixed sheaves'' on flag varieties (see e.g.~\cite[Discussion following Definition~1.2]{cvdhs}), and might deserve a better understanding.

This construction is closely related to that of a ``Koszul duality'' for categories of perverse sheaves on flag varieties as in~\cite{bgs,by,amrw}, and in fact in~\S\ref{ss:Koszul-duality} we explain how our methods allow (again following the methods of~\cite{modrap2}) to provide an alternative construction of the ``modular'' version of this duality from~\cite{amrw}, in the special case of affine flag varieties.

In the final~\S\ref{ss:Whit-par} we also discuss Whittaker and parahoric variants of the problems studied above.

\begin{rmk}
 There is another very important family of perverse sheaves on $\Fl_G$, namely the ``central perverse sheaves'' associated with representations of $G^\vee_\bk$; see e.g.~\cite{ar-book}. Versions of these objects can now also be defined in the mixed setting, thanks to work of Cass--van den Hove--Scholbach~\cite{cvdhs}. It is likely that our functor $\nu$ sends mixed central perverse sheaves to central perverse sheaves, but this problem will not be studied here.
\end{rmk}

\subsection{Acknowledgements}

The present work is the continuation of a long-term collaboration with P.~Achar, and an outgrowth of a joint work with R.~Bezrukav\-nikov. We thank both of them for their collaboration and for sharing their ideas on this and related subjects.

\section{Some Hecke categories}
\label{sec:Hecke-cat}

\subsection{Affine and extended affine Weyl groups}
\label{ss:Waff}


Let $\bk$ be an algebraically closed field of characteristic $p$, and $\bG$ be a connected reductive algebraic group over $\bk$. We fix a Borel subgroup $\bB \subset \bG$ and a maximal torus $\bT \subset \bB$, and denote by $\bWf$ the Weyl group of $(\bG,\bT)$, i.e.~the quotient $\mathrm{N}_{\bG}(\bT)/\bT$. Then $\bWf$ is a finite group, and the choice of $\bB$ determines a system of generators $\bSf \subset \bWf$ such that $(\bWf,\bSf)$ is a Coxeter system. More specifically, consider the character lattice $\rmX^*(\bT)$ of $\bT$, and the subset of roots $\fR \subset \rmX^*(\bT)$. We denote by $\fR_+ \subset \fR$ the positive system consisting of the opposites of the $\bT$-weights in the Lie algebra of $\bB$, and by $\fR_{\mathrm{s}}$ the corresponding basis of $\fR$. Then $\bSf$ consists of the reflections associated with the simple roots. We will also consider the cocharacter lattice $\rmX_*(\bT)$, and the system of coroots $\fR^\vee \subset \rmX_*(\bT)$. The Lie algebras of $\bG$, $\bB$, $\bT$ will be denoted $\bg$, $\bb$, $\bt$ respectively.

The extended affine Weyl group of $(\bG,\bT)$ is the semidirect product
\[
\bW := \bWf \ltimes \rmX^*(\bT).
\]
The affine Weyl group of $(\bG,\bT)$ is the subgroup
\[
\bWaff := \bWf \ltimes \Z\fR
\]
where $\Z\fR \subset \rmX^*(\bT)$ is the lattice generated by $\fR$.
For $\lambda \in \rmX^*(\bT)$, we will denote by $\st(\lambda)$ the associated element of $\bW$.
It is a standard fact that there exists a natural subset $\bSaff \subset \bWaff$ containing $\bSf$ and such that $(\bWaff,\bSaff)$ is a Coxeter system; more precisely $\bSaff$ consists of the elements of $\bSf$ together with the products $\st(\beta) s_\beta$ where $\beta$ is a maximal short root. By construction, $\bWf$ is then the parabolic subgroup of $\bWaff$ generated by $\bSf$.

If we set, for $w \in \bWf$ and $\lambda \in \rmX^*(\bT)$,
\begin{equation}
\label{eqn:formula-length}
\ell(w\st(\lambda))= \sum_{\substack{\alpha \in \mathfrak{R}_+ \\ w(\alpha) \in \mathfrak{R}_+}} |\langle \lambda, \alpha^\vee \rangle| + \sum_{\substack{\alpha \in \mathfrak{R}_+ \\ w(\alpha) \in -\mathfrak{R}_+}} |\langle \lambda, \alpha^\vee \rangle+1|,
\end{equation}
then it is well known that the restriction of $\ell$ to $\bWaff$ is the length function associated with our Coxeter generators $\bSaff$, and that if we set $\mathbf{\Omega} = \{w \in \bW \mid \ell(w)=0\}$ then the natural morphism
\[
\mathbf{\Omega} \ltimes \bWaff \to \bW
\]
is a group isomorphism. Moreover, in this semidirect product $\mathbf{\Omega}$ acts on $\bWaff$ by Coxeter group automorphisms, i.e.~it stabilizes $\bSaff$.

Recall from~\cite[Lemma~3.1]{reg-quotient-pt2} that if $\bG$ has simply connected derived subgroup, for any $s \in \bSaff \smallsetminus \bSf$, there exist $s' \in \bSf$ and $w \in \bW$ such that $\ell(ws')=\ell(w)+1$ and $s=ws'w^{-1}$.
We will fix once and for all such elements.

\subsection{Representations of the universal centralizer group scheme}
\label{ss:representations-Iadj}

Recall (see e.g.~\cite{reg-quotient-pt2}) that for any separated $\bk$-scheme $X$ endowed with an action of $\bG$ we can consider the associated universal stabilizer, defined as the fiber product
\[
(\bG \times X) \times_{X \times X} X
\]
where the morphism $\bG \times X \to X \times X$ is given by $(g,x) \mapsto (g \cdot x,x)$, and the morphism $X \to X \times X$ is the diagonal embedding. This scheme has a natural structure of affine group scheme over $X$ (with respect to the natural projection to $X$).

One can in particular consider this construction for the adjoint action of $\bG$ on itself; the associated group scheme is denoted $\bbJ$, and called the \emph{universal centralizer}. This group scheme itself is not so interesting because it is not flat, but assuming that the following conditions hold:
\begin{enumerate}
\item
 $\bG$ has simply connected derived subgroup;
\item
\label{it:ass-center}
 the scheme-theoretic center $\mathrm{Z}(\bG) \subset \bG$ is smooth,
 \end{enumerate}
 the restriction $\bbJ_{\mathrm{reg}}$ of $\bbJ$ to the open subscheme $\bG_{\mathrm{reg}} \subset \bG$ of regular elements is smooth. (This statement is classical; for a proof in this generality, see~\cite[Lemma~2.17]{reg-quotient-pt2}.) From now on we assume that these conditions are satisfied. (Note for later use that~\eqref{it:ass-center} is equivalent to the property that $\rmX^*(\bT)/\Z\fR$ has no $p$-torsion; see~\cite[\S 2.3]{reg-quotient-pt2}.)

Consider a ``Steinberg section'' $\mathbf{\Sigma}$ as in~\cite[\S 2.2]{reg-quotient-pt2}. The properties of this section that we will need are the following: $\mathbf{\Sigma}$ is a closed subscheme in $\bG$, contained in $\bG_{\mathrm{reg}}$, and the composition
\begin{equation}
\label{eqn:map-Sigma-adj-quotient}
\mathbf{\Sigma} \hookrightarrow \bG \to \bG/\bG = \bT/\bWf
\end{equation}
is an isomorphism. (Here the second morphism is the adjoint quotient map, and the identification on the right-hand side is classical.)

Let $\bbJ_{\mathbf{\Sigma}}$ be the restriction of $\bbJ$ to $\mathbf{\Sigma}$, a smooth affine group scheme over $\mathbf{\Sigma}$, and set
\[
\bbI_{\mathbf{\Sigma}} = (\bT \times_{\bT/\bWf} \bT) \times_{\bT/\bWf} \bbJ_{\mathbf{\Sigma}},
\]
where the map $\bbJ_{\mathbf{\Sigma}} \to \bT/\bWf$ is the composition of the projection $\bbJ_{\mathbf{\Sigma}} \to \mathbf{\Sigma}$ with the isomorphism~\eqref{eqn:map-Sigma-adj-quotient}. Let also
$\Rep(\bbI_{\mathbf{\Sigma}})$
be the abelian category
of representations of $\bbI_{\mathbf{\Sigma}}$ on coherent $\scO_{\bT \times_{\bT/\bWf} \bT}$-modules.
It identifies with the category of comodules over the $\scO(\bT \times_{\bT/\bWf} \bT)$-Hopf algebra
\[
\scO(\bbI_{\mathbf{\Sigma}}) = \scO(\bbJ_{\mathbf{\Sigma}}) \otimes_{\scO(\mathbf{\Sigma})} \scO(\bT \times_{\bT/\bWf} \bT)
\]
which are finitely generated as $\scO(\bT \times_{\bT/\bWf} \bT)$-modules. Since $\scO(\bT \times_{\bT/\bWf} \bT)$ is finite as an $\scO(\mathbf{\Sigma})$-module, this category admits a natural monoidal structure defined by
\[
M \circledast N = M \otimes_{\scO(\bT)} N.
\]
This bifunctor is right exact on each side, and the unit object for this monoidal structure is $\scO(\bT)$, seen as functions on the diagonal copy $\bT \subset \bT \times_{\bT/\bWf} \bT$, and endowed with the trivial structure as a representation of $\bbI_{\mathbf{\Sigma}}$.

We will now define (following~\cite[\S 3.2]{reg-quotient-pt2}) objects $(\mathscr{M}_w : w \in \bW)$ of $\Rep(\bbI_{\mathbf{\Sigma}})$ parametrized by $\bW$ as follows. First, if $w \in \bWf$ then $\mathscr{M}_w$ is defined as the structure sheaf of the closed subscheme
\[
\{(w(t),t) : t \in \bT\} \subset \bT \times_{\bT/\bWf} \bT,
\]
endowed with the trivial structure as a representation.
The projection on the first component induces an isomorphism $\mathscr{M}_w \simto \scO(\bT)$; under this isomorphism, the action of $\scO(\bT \times_{\bT/\bWf} \bT)=\scO(\bT) \otimes_{\scO(\bT/\bWf)} \scO(\bT)$ on $\mathscr{M}_w$ is given by $(f \otimes g) \cdot m = fw(g) m$ for $f,g,m \in \scO(\bT)$. 

By~\cite[Equation~(2.13)]{reg-quotient-pt2}, there exists a canonical morphism of group schemes
\[
\bbJ_{\mathbf{\Sigma}} \times_{\bT/\bWf} \bT \to \bT \times \bT,
\]
where the right-hand side is seen as a group scheme over $\bT$ via the first projection.
If $\lambda \in \rmX^*(\bT)$, one can consider the representation $\scO_{\bT} \otimes \bk_{\bT}(\lambda)$ of $\bT \times \bT$. Restricting this representation to $\bbJ_{\mathbf{\Sigma}} \times_{\bT/\bWf} \bT$, and then pushing the result forward along the diagonal embedding $\bT \to \bT \times_{\bT/\bWf} \bT$ we obtain an object of $\Rep(\bbI_{\mathbf{\Sigma}})$, which will be denoted $\mathscr{M}_{\st(\lambda)}$.

It is clear that for $w,y \in \bWf$ and $\lambda,\mu \in \rmX^*(\bT)$ we have canonical isomorphisms
\begin{align}
\label{eqn:formula-M-1}
\mathscr{M}_w \circledast \mathscr{M}_y & \simto \mathscr{M}_{wy}, \\
\label{eqn:formula-M-2}
\mathscr{M}_{\st(\lambda)} \circledast \mathscr{M}_{\st(\mu)} &\simto \mathscr{M}_{\st(\lambda+\mu)}.
\end{align}
It is explained in the discussion following~\cite[Lemma~3.2]{reg-quotient-pt2} that for $w \in \bWf$ and $\lambda \in \rmX^*(\bT)$ we have a canonical isomorphism
\[
\mathscr{M}_w \circledast \mathscr{M}_{\st(\lambda)} \circledast \mathscr{M}_{w^{-1}} \simto \mathscr{M}_{\st(w(\lambda))}.
\]
Combining this with~\eqref{eqn:formula-M-1}--\eqref{eqn:formula-M-2} we deduce that if for $w=x\st(\lambda) \in \bWf \ltimes \rmX^*(\bT)=\bW$ we set
\[
\mathscr{M}_w := \mathscr{M}_x \circledast \mathscr{M}_{\st(\lambda)},
\]
then for any $w,y \in \bW$ we have a canonical isomorphism 
\[
\mathscr{M}_w \circledast \mathscr{M}_y \simto \mathscr{M}_{wy}.
\]

We next define some objects $(\mathscr{B}_s : s \in \bSaff)$ associated with simple reflections in $\bWaff$. First, if $s \in \bSf$ we define $\mathscr{B}_s$ by
\[
\mathscr{B}_s := \scO(\bT \times_{\bT/\{1,s\}} \bT),
\]
which we view as an $\scO(\bT \times_{\bT/\bWf} \bT)$-module via the closed embedding
\[
\bT \times_{\bT/\{1,s\}} \bT \subset \bT \times_{\bT/\bWf} \bT,
\]
and endow with the trivial structure as a representation. If $s \in \bSaff \smallsetminus \bSf$, recall from~\S\ref{ss:Waff} that we have fixed $s' \in \bSf$ and $w \in \bW$ such that $s=ws'w^{-1}$; we set
\begin{equation}
\label{eqn:def-Bs}
\mathscr{B}_s := \mathscr{M}_w \circledast \mathscr{B}_{s'} \circledast \mathscr{M}_{w^{-1}}.
\end{equation}
It is easily seen (e.g.~by reduction to the case $s \in \bSf$) that for any $s \in \bSaff$ there exist exact sequences
\[
\mathscr{M}_e \hookrightarrow \mathscr{B}_s \twoheadrightarrow \mathscr{M}_s, \quad \mathscr{M}_s \hookrightarrow \mathscr{B}_s \twoheadrightarrow \mathscr{M}_e.
\]

 \subsection{Completions}
 \label{ss:completions}

Consider the ideal in $\scO(\bT \times_{\bT/\bWf} \bT)$ defined by the point $(e,e)$ (where $e \in \bT$ is the neutral element), and denote
by $(\bT \times_{\bT/\bWf} \bT)^\wedge$ the spectrum of the completion of $\scO(\bT \times_{\bT/\bWf} \bT)$ with respect to this ideal. We define similarly $\bT^\wedge$ and $(\bT/\bWf)^\wedge$ as the spectra of the completions of $\scO(\bT)$ and $\scO(\bT/\bWf)$ with respect to the ideal corresponding to $e$ and its image, respectively. Then, as explained in~\cite[Lemma~3.3]{reg-quotient-pt2}, we have canonical isomorphisms
\begin{equation}
\label{eqn:isom-completions}
(\bT \times_{\bT/\bWf} \bT)^\wedge \cong (\bT \times_{\bT/\bWf} \bT) \times_{\bT/\bWf} (\bT/\bWf)^\wedge \cong \bT^\wedge \times_{(\bT/\bWf)^\wedge} \bT^\wedge.
\end{equation}
It is also proved in \emph{loc.}~\emph{cit.} that the natural morphism $\scO((\bT/\bWf)^\wedge) \to \scO(\bT^\wedge)$ induces an isomorphism
\begin{equation}
\label{eqn:isom-completions-2}
\scO((\bT/\bWf)^\wedge) \simto \scO(\bT^\wedge)^{\bWf}.
\end{equation}

We set
\[
\bbI_{\mathbf{\Sigma}}^\wedge := (\bT \times_{\bT/\bWf} \bT)^\wedge \times_{\bT \times_{\bT/\bWf} \bT} \bbI_{\mathbf{\Sigma}} \cong (\bT \times_{\bT/\bWf} \bT)^\wedge \times_{\bT / \bWf} \mathbb{J}_{\mathbf{\Sigma}},
\]
a smooth affine group scheme over the affine scheme $(\bT \times_{\bT/\bWf} \bT)^\wedge$. We will consider the category $\Rep(\bbI_{\mathbf{\Sigma}}^\wedge)$ of representations of this group scheme on coherent $\scO_{(\bT \times_{\bT/\bWf} \bT)^\wedge}$-modules. 
The isomorphisms in~\eqref{eqn:isom-completions} show that an $\scO((\bT \times_{\bT/\bWf} \bT)^\wedge)$-module is the same thing as an $\scO(\bT^\wedge)$-bimodule on which the left and right actions of $\scO((\bT/\bWf)^\wedge)$ coincide. (We will use this identification repeatedly and without further notice below.) In particular the category of such modules admits a natural monoidal product, induced by the tensor product for $\scO(\bT^\wedge)$-bimodules; moreover this product stabilizes the subcategory of finitely generated $\scO((\bT \times_{\bT/\bWf} \bT)^\wedge)$-modules. Since $\bbI_{\mathbf{\Sigma}}^\wedge$ is the pullback of a group scheme over $(\bT/\bWf)^\wedge$, this product induces a monoidal product on the category $\Rep(\bbI_{\mathbf{\Sigma}}^\wedge)$, which will again be denoted $\circledast$. In~\cite[Lemma~3.4]{reg-quotient-pt2} it is also proved that the category $\Rep(\bbI_{\mathbf{\Sigma}}^\wedge)$ is Krull--Schmidt.

%


Pulling back the representations $(\mathscr{M}_w : w \in \bW)$ and $(\mathscr{B}_s : s \in \bSaff)$ introduced in~\S\ref{ss:representations-Iadj} along
the natural morphism $(\bT \times_{\bT/\bWf} \bT)^\wedge \to \bT \times_{\bT/\bWf} \bT$ we obtain objects $(\mathscr{M}^\wedge_w : w \in \bW)$ and $(\mathscr{B}^\wedge_s : s \in \bSaff)$ in $\Rep(\bbI^\wedge_{\mathbf{\Sigma}})$. It is clear that for any $w,y \in \bW$ we have a canonical isomorphism 
\begin{equation}
\label{eqn:convolution-D}
\mathscr{M}^\wedge_w \circledast \mathscr{M}^\wedge_y \simto \mathscr{M}^\wedge_{wy},
\end{equation}
and that for $s \in \bSaff$ we have exact sequences
\begin{equation}
 \label{eqn:exact-seq-Bs}
 \mathscr{M}^\wedge_e \hookrightarrow \mathscr{B}^\wedge_s \twoheadrightarrow \mathscr{M}^\wedge_s, \qquad \mathscr{M}^\wedge_s \hookrightarrow \mathscr{B}^\wedge_s \twoheadrightarrow \mathscr{M}^\wedge_e.
 \end{equation}
 By~\cite[Lemma~3.5]{reg-quotient-pt2}, 
for any $s \in \bSaff \smallsetminus \bSf$ the object $\scB^\wedge_s$ is independent of the choices of $w$ and $s'$ as in~\S\ref{ss:Waff} up to canonical isomorphism; moreover, for any $\omega \in \mathbf{\Omega}$ and $s \in \bSaff$ we have a canonical isomorphism
 \begin{equation}
 \label{eqn:conjugation-Bs}
  \scM^\wedge_\omega \circledast \scB^\wedge_s \circledast \scM^\wedge_{\omega^{-1}} \cong \scB^\wedge_{\omega s \omega^{-1}}.
 \end{equation}

We will denote by $\BSRep(\bbI_{\mathbf{\Sigma}}^\wedge)$ the category with
\begin{itemize}
\item
objects the collections $(\omega, s_1, \dots, s_i)$ with $\omega \in \mathbf{\Omega}$ and $s_1, \dots, s_i \in \bSaff$;
\item
morphisms from $(\omega, s_1, \dots, s_i)$ to $(\omega', s'_1, \dots, s'_j)$ given by
\[
\Hom_{\Rep(\bbI_{\mathbf{\Sigma}}^\wedge)}(\scM^\wedge_\omega \circledast \scB^\wedge_{s_1} \circledast \cdots \circledast \scB^\wedge_{s_i}, \scM^\wedge_{\omega'} \circledast \scB^\wedge_{s_1'} \circledast \cdots \circledast \scB^\wedge_{s'_j}).
\]
\end{itemize}
By definition there exists a canonical fully faithful functor
\begin{equation}
\label{eqn:functor-BSRep-Rep}
\BSRep(\bbI_{\mathbf{\Sigma}}^\wedge) \to \Rep(\bbI_{\mathbf{\Sigma}}^\wedge).
\end{equation}
For any collections $(\omega, s_1, \dots, s_i)$ and $(\omega', s'_1, \dots, s'_j)$ as above, we have a canonical isomorphism
\begin{multline*}
 \bigl( \scM^\wedge_\omega \circledast \scB^\wedge_{s_1} \circledast \cdots \circledast \scB^\wedge_{s_i} \bigr) \circledast \bigl( \scM^\wedge_{\omega'} \circledast \scB^\wedge_{s_1'} \circledast \cdots \circledast \scB^\wedge_{s'_j} \bigr) \\
 \cong \scM^\wedge_{\omega \omega'} \circledast \scB^\wedge_{(\omega')^{-1} s_1 \omega'} \circledast \cdots \circledast \scB^\wedge_{(\omega')^{-1} s_i \omega'} \circledast \scB^\wedge_{s_1'} \circledast \cdots \circledast \scB^\wedge_{s'_j};
\end{multline*}
this allows us to define a monoidal product (again denoted $\circledast$) on $\BSRep(\bbI_{\mathbf{\Sigma}}^\wedge)$ which is defined on objects by
\[
 (\omega, s_1, \dots, s_i) \circledast (\omega', s'_1, \dots, s'_j) = (\omega\omega', (\omega')^{-1} s_1 \omega', \dots, (\omega')^{-1} s_i \omega', s'_1, \dots, s'_j)
\]
and such that~\eqref{eqn:functor-BSRep-Rep} is monoidal.

We will denote by
\[
\SRep(\bbI^\wedge_{\mathbf{\Sigma}})
\]
the karoubian closure of the additive hull of the category $\BSRep(\bbI_{\mathbf{\Sigma}}^\wedge)$. By the Krull--Schmidt property, this category identifies with the (monoidal) full subcategory of $\Rep(\bbI^\wedge_{\mathbf{\Sigma}})$ whose objects are direct sums of direct summands of objects of the form
\[
 \scM^\wedge_\omega \circledast \scB^\wedge_{s_1} \circledast \cdots \circledast \scB^\wedge_{s_i}
\]
with $\omega \in \mathbf{\Omega}$ and $s_1, \dots, s_i \in \bSaff$. (In these notations, ``$\mathsf{BS}$'' stands for ``Bott--Samelson,'' and ``$\mathsf{S}$'' for ``Soergel,'' since these constructions are very similar to classical constructions related to Bott--Samelson resolutions and Soergel bimodules.)

\subsection{Hecke categories ``\`a la Abe''}
\label{ss:Hecke-a-la-Abe}

We now explain how to construct some categories following a pattern initiated by Abe~\cite{abe}, see also~\cite[\S 3.4]{reg-quotient-pt2}. We consider a noetherian domain $R$ endowed with an action of $\bW$ (by ring automorphisms), and denote by $Q$ the fraction field of $R$. We denote by $\sfK'(R)$ the category defined as follows. The objects are the $R$-bimodules $M$ together with a decomposition
\begin{equation}
\label{eqn:decomp-K-Abe}
M \otimes_{R} Q = \bigoplus_{w \in \bW} M^w_{Q}
\end{equation}
as $(R, Q)$-bimodules such that:
\begin{itemize}
\item there exist only finitely many $w$'s such that $M^w_{Q} \neq 0$;
\item for any $w \in \bW$, $r \in R$ and $m \in M^w_{Q}$ we have $m \cdot r = w(r) \cdot m$.
\end{itemize}
Morphisms in this category are defined as morphisms of $R$-bimodules respecting the decompositions~\eqref{eqn:decomp-K-Abe}. The category $\sfK^{\prime}(R)$ has a natural monoidal structure, with product denoted $\star$ and induced by the tensor product over $R$. (To see this one observes that the conditions above imply that the left $R$-action on $M \otimes_{R} Q$ extends to an action of $Q$, see~\cite[Remark~2.2]{abe}.)

We will also denote by $\sfK(R)$ the full subcategory in $\sfK'(R)$ whose objects are those whose underlying $R$-bimodule is finitely generated, and which are flat as right $R$-modules. The latter condition implies that the natural morphism $M \to M \otimes_R Q$ is injective, which (in view of the second condition above) implies in particular that the left and right actions of $R^\bW$ on $M$ coincide. The arguments in~\cite[Lemma~2.6]{abe} show that the underlying bimodule of any object in $\sfK(R)$ is in fact finitely generated as a left $R$-module and as a right $R$-module. Using this property, it is easily seen that $\sfK(R)$ is a monoidal subcategory of $\sfK'(R)$.

\begin{rmk}
\label{rmk:switch-abe}
As explained in~\cite[\S 2.2]{abe}, for any $M$ in $\sfK'(R)$ there exists a canonical isomorphism $Q \otimes_R M \simto M \otimes_R Q$. 
(In the examples we will consider below, the action of $\bW$ on $R$ will factor through an action of the finite group $\bWf$, so that $R$ will be finite over $R^\bW$. In this case, both $Q \otimes_R M$ and $M \otimes_R Q$ identify with $M \otimes_{R^\bW} \mathrm{Frac}(R^\bW)$.)
As a consequence, switching the left and right $R$-actions defines an autoequivalence of the category $\sfK'(R)$, where the $w$-graded part in the image of $M$ is $M_Q^{w^{-1}}$ with the actions switched. This equivalence is ``antimonoidal'' in the sense that it swaps factors in a tensor product. It restricts to an autoequivalence of the subcategory of $\sfK'(R)$ whose objects are finitely generated (as bimodules) and flat both as a left and as a right $R$-module.
\end{rmk}

We have natural objects in $\sfK(R)$ attached to elements in $\bW$, and constructed as follows. Given $w \in \bW$, we denote by $F_w$ the $R$-bimodule which is isomorphic to $R$ as an abelian group, and endowed with the structure of $R$-bimodule determined by the rule
\[
 r \cdot m \cdot r' = rmw(r')
\]
for $r,r' \in R$ and $m \in F_w$. If we endow this bimodule with the decomposition of $F_{w} \otimes_{R} Q$ such that this module is concentrated in degree $w$, we obtain an object in $\sfK(R)$. It is clear that for any $w,y \in \bW$ we have a canonical isomorphism
\[
F_w \star F_y \simto F_{wy}.
\]

Next, for $s \in \bS$ we will denote by $R^s \subset R$ the subring of $s$-invariants. Assume that 
\begin{equation}
\label{eqn:assumption-deltas}
\text{there exists $\delta _s \in R$ such that $(1,\delta_s)$ is a basis of $R$ as an $R^s$-module.}
\end{equation}
Then we set
\[
B_{s} := R \otimes_{R^s} R.
\]
Our assumption ensures that $B_s$ is finite and free (in particular, flat) as a right $R$-module.
Moreover this objects admits a canonical decomposition~\eqref{eqn:decomp-K-Abe}, hence defines an object in $\sfK(R)$. In fact,
since the action of $s$ on $R$ is nontrivial by our assumption, the decomposition of $B_{s} \otimes_{R} Q = R \otimes_{R^s} Q$ is uniquely determined by the fact that it is concentrated in degrees $\{e,s\} \subset \bW$.
More explicitly, using the formula 
\[
\delta_s \delta_s = \delta_s (\delta_s + s(\delta_s)) - \delta_s s(\delta_s)
\]
one checks that we have
\[
(B_s)_Q^e = (\delta_s \otimes 1 - 1 \otimes s(\delta_s)) \cdot Q, \quad (B_s)_Q^s = (\delta_s \otimes 1 - 1 \otimes \delta_s) \cdot Q.
\]

By~\cite[Lemma~3.6]{reg-quotient-pt2}, if $s,s' \in \bSaff$ and $w \in \bW$ satisfy $s'=wsw^{-1}$, and if~\eqref{eqn:assumption-deltas} holds for $s$, then this condition also holds for $s'$, and moreover we have a canonical isomorphism
\begin{equation}
\label{eqn:conj-Fw-Bs}
F_w \star B_s \star F_{w^{-1}} \simto B_{s'}.
\end{equation}

We now assume that~\eqref{eqn:assumption-deltas} is satisfied for any $s \in \bSaff$. We will then denote by $\BSK(R)$ the category with
\begin{itemize}
\item
objects the collections $(\omega, s_1, \dots, s_i)$ with $\omega \in \mathbf{\Omega}$ and $s_1, \dots, s_i \in \bSaff$;
\item
morphisms from $(\omega, s_1, \dots, s_i)$ to $(\omega', s'_1, \dots, s'_j)$ given by
\[
\Hom_{\sfK(R)}(F_\omega \star B_{s_1} \star \cdots \star B_{s_i}, F_{\omega'} \star B_{s_1'} \star \cdots \star B_{s'_j}).
\]
\end{itemize}
By definition there exists a canonical fully faithful functor
\begin{equation}
\label{eqn:functor-HBS-K}
\BSK(R) \to \sfK(R).
\end{equation}
Using the isomorphisms~\eqref{eqn:conj-Fw-Bs} (when $w \in \mathbf{\Omega}$) one sees that there exists a natural convolution product (still denoted $\star$) on $\BSK(R)$ which is defined on objects by
\[
 (\omega, s_1, \dots, s_i) \star (\omega', s'_1, \dots, s'_j) = (\omega\omega', (\omega')^{-1} s_1 \omega', \dots, (\omega')^{-1} s_i \omega', s'_1, \dots, s'_j),
\]
and such that~\eqref{eqn:functor-HBS-K} is monoidal.

\begin{rmk}
\label{rmk:switch-abe-BS}
Instead of putting the element in $\mathbf{\Omega}$ to the left, one can also put it to the right, and define the monoidal category $\BSKr(R)$ with objects the collections $(s_1, \dots, s_i, \omega)$ and morphisms defined in the obvious way.
The equivalence of Remark~\ref{rmk:switch-abe} sends each $B_s$ to itself, and each $F_w$ to $F_{w^{-1}}$. 
It therefore induces an equivalence of categories $\BSK(R) \simto \BSKr(R)$ which is antimonoidal and is given on objects by
\[
(\omega, s_1, \dots, s_i) \mapsto (s_i, \dots, s_1, \omega^{-1}).
\] 
The same comment applies to the graded versions introduced in~\S\ref{ss:graded-Hecke} below.
\end{rmk}


The following lemma is obvious.

\begin{lem}
\label{lem:isom-equiv-Abe-cat}
Let $R$ and $R'$ be two noetherian domains endowed with actions of $\bW$ by ring automorphisms. Assume that condition~\eqref{eqn:assumption-deltas} is satisfied for the ring $R$ (for any $s \in \bSaff$), and assume given a $\bW$-equivariant ring isomorphism $R \simto R'$. Then condition~\eqref{eqn:assumption-deltas} is satisfied for the ring $R'$ (for any $s \in \bSaff$), and there exists a natural equivalence of monoidal categories
\[
\sfK'(R) \simto \sfK'(R')
\]
which restricts to an equivalence
\[
\sfK(R) \simto \sfK(R')
\]
sending each object $F_w$ ($w \in \bW$) or $B_s$ ($s \in \bSaff$) in $\sfK(R)$ to the corresponding object in $\sfK(R')$. As a consequence, we deduce an equivalence of categories
\[
\BSK(R) \simto \BSK(R')
\]
which is the identity on objects.
\end{lem}

\subsection{Graded Hecke categories}
\label{ss:graded-Hecke}

The construction of~\S\ref{ss:Hecke-a-la-Abe} admits a ``graded variant'' as follows. In this setting we assume that $R$ is a ($\mathbb{Z}$-)\emph{graded} noetherian domain, and that the action of $\bW$ is by graded ring automorphisms. Then we have a ``grading shift'' functor $(1)$ on graded $R$-bimodules, defined in such a way that $(M(1))^i = M^{i+1}$ for any $i \in \mathbb{Z}$. In this setting we define the category $\sfK'_{\mathrm{gr}}(R)$ as above, but using \emph{graded} $R$-bimodules and morphisms of graded bimodules. (The fraction field $Q$ has no grading, and we impose no compatibility of the decomposition of $M \otimes_R Q$ with the grading.) One then defines the subcategory $\sfK_{\mathrm{gr}}(R)$ in the same way as above. In order to define $B_s$ we assume that there exists a homogeneous element $\delta_s \in R$ such that $(1,\delta_s)$ is a basis of $R$ as an $R^s$-module. Moreover, we set
\[
B_{s} := R \otimes_{R^s} R (1).
\]
Finally, $\BSK_{\mathrm{gr}}(R)$ is defined as the category with: 
\begin{itemize}
\item
objects the collections $(\omega, s_1, \dots, s_i, n)$ with $\omega \in \mathbf{\Omega}$ and $s_1, \dots, s_i \in \bSaff$ and $n \in \Z$;
\item
morphisms from $(\omega, s_1, \dots, s_i,n)$ to $(\omega', s'_1, \dots, s'_j, n')$ given by
\[
\Hom_{\sfK_{\mathrm{gr}}(R)}(F_\omega \star B_{s_1} \star \cdots \star B_{s_i}(n), F_{\omega'} \star B_{s_1'} \star \cdots \star B_{s'_j}(n')).
\]
\end{itemize}
As above we have a canonical fully faithful functor $\BSK_{\mathrm{gr}}(R) \to \sfK_{\mathrm{gr}}(R)$.
Given $M,N$ in $\sfK_{\mathrm{gr}}(R)$, we will set
\[
\Hom^\bullet_{\sfK_{\mathrm{gr}}(R)}(M,N) = \bigoplus_{n \in \Z} \Hom_{\sfK_{\mathrm{gr}}(R)}(M,N(n)).
\]
Again the category $\sfK'_{\mathrm{gr}}(R)$ admits a natural convolution product $\star$, which makes it a monoidal category and stabilizes the subcategory $\sfK_{\mathrm{gr}}(R)$, and which induces a monoidal structure on $\BSK_{\mathrm{gr}}(R)$ given on objects by
\begin{multline*}
 (\omega, s_1, \dots, s_i, n) \star (\omega', s'_1, \dots, s'_j, n') = \\
 (\omega\omega', (\omega')^{-1} s_1 \omega', \dots, (\omega')^{-1} s_i \omega', s'_1, \dots, s'_j, n+n').
\end{multline*}

\subsection{Completed Hecke category and representations of \texorpdfstring{$\bbI_{\mathbf{\Sigma}}^\wedge$}{ISigma}}
\label{ss:completed-Hecke-Rep}

The first ring to which we will apply the construction of~\S\ref{ss:Hecke-a-la-Abe} is $\scO(\bT^\wedge)$, with the action of $\bW$ obtained from the natural action of $\bWf$ by pullback along the projection $\bW \to \bWf$. 
It is explained in~\cite[\S 3.5]{reg-quotient-pt2} that the condition~\eqref{eqn:assumption-deltas} is satisfied for any $s \in \bSaff$ in this case.
The resulting categories $\sfK(\scO(\bT^\wedge))$ and $\BSK(\scO(\bT^\wedge))$ will be denoted
\[
\sfK^\wedge \quad \text{ and } \quad \BSK^{\wedge}
\]
respectively.

Recall the category $\Rep(\bbI_{\mathbf{\Sigma}}^\wedge)$ considered in~\S\ref{ss:completions}. We will denote by
$\Rep_{\mathrm{fl}}(\bbI_{\mathbf{\Sigma}}^\wedge)$ the full subcategory of representations whose underlying coherent sheaf is flat with respect to the projection $(\bT \times_{\bT/\bWf} \bT)^\wedge \to \bT^\wedge$ on the second component. It is not difficult to check that $\Rep_{\mathrm{fl}}(\bbI_{\mathbf{\Sigma}}^\wedge)$ is a monoidal subcategory in $\Rep(\bbI_{\mathbf{\Sigma}}^\wedge)$, and that it contains the essential image of~\eqref{eqn:functor-BSRep-Rep}.

With this definition, it is proved in~\cite[Proposition~3.9]{reg-quotient-pt2} that
there exists a canonical fully faithful monoidal functor
\[
\Rep_{\mathrm{fl}}(\bbI_{\mathbf{\Sigma}}^\wedge) \to \sfK^{\wedge}
\]
sending $\scM_w^\wedge$ to $F_w$ for any $w \in \bW$ and $\scB^\wedge_s$ to $B_s$ for any $s \in \bSaff$.
By monoidality, for any $\omega \in \mathbf{\Omega}$ and $s_1, \dots, s_i \in \bS$ this functor sends $\scM^\wedge_{\omega} \circledast \scB^\wedge_{s_1} \circledast \cdots \circledast \scB^\wedge_{s_i}$
to $F_\omega \star B_{s_1} \star \cdots \star B_{s_i}$; it therefore induces 
an equivalence of monoidal categories
\begin{equation}
\label{eqn:equiv-BSK-BSRep}
 \BSRep(\bbI_{\mathbf{\Sigma}}^\wedge) \simto \BSK^\wedge
\end{equation}
which is the identity on objects.

\subsection{``Additive'' Hecke categories}
\label{ss:additive-Hecke}

From now on, in addition to our running assumptions (see~\S\ref{ss:representations-Iadj}) we will assume that:
\begin{enumerate}
 \item
 $p$ is good for $\bG$;
 \item
 \label{it:bilin-form}
 there exists a $\bG$-equivariant isomorphism $\bg \simto \bg^*$ (which we fix from now on).
\end{enumerate}

\begin{rmk}
\label{rmk:assumptions-Hecke-0}
The first assumption is explicit and mild; the second one holds in particular if $\bG=\mathrm{GL}_n$, and if $p$ is very good for $\bG$, see~\cite[Proposition~2.5.12]{letellier}.
\end{rmk}

We consider the ring $\scO(\bt^*)$, endowed with the grading such that $\bt \subset \scO(\bt^*)$ is placed in degree $2$, and with the action of $\bW$ obtained from the natural action of $\bWf$ by pullback along the projection $\bW \to \bWf$. Conditions~\eqref{eqn:assumption-deltas} are again satisfied in this case; indeed by \cite[Remark~3.7]{reg-quotient-pt2} we can assume that $s \in \bSf$. In this case, if $\alpha$ is the associated simple root, as explained in~\cite[Claim 3.11]{ew} one can take as $\delta_s$ any element $x \in \bt$ such that $d(\alpha)(x)=1$ where $d(\alpha)$ is the differential of $\alpha$. (Such an element does exist since $\rmX^*(\bT)/\Z\fR$ has no $p$-torsion.) The categories $\sfK_{\mathrm{gr}}(\scO(\bt^*))$ and $\BSK_{\mathrm{gr}}(\scO(\bt^*))$ will be denoted
\[
\sfK_\add \quad \text{ and } \quad \BSK_\add.
\]

\begin{rmk}
 The categories $\sfK_\add$ and $\BSK_\add$ are (up to the subtleties related to length-$0$ elements) the categories denoted $\mathcal{C}$ and $\mathcal{BS}$ in~\cite{abe}, for the following data:
\begin{itemize}
 \item the underlying $\bk$-vector space is $V=\bt$;
 \item if $s \in \bSf$, and if $\alpha$ is the simple root associated with $s$, then the ``root'' $\alpha_s \in \bt$ is the differential of $\alpha^\vee$, and the ``coroot'' $\alpha_s^\vee \in \bt^*$ is the differential of $\alpha$;
 \item if $\beta \in \mathfrak{R}_+$ is a maximal short root and $s=\st(\beta) s_\beta$, then the ``root'' $\alpha_s \in \bt$ is the opposite of the differential of $\beta^\vee$, and the ``coroot'' $\alpha_s^\vee \in \bt^*$ is the opposite of the differential of $\beta$. 
\end{itemize}
(As explained in~\cite[\S 2.2]{brHecke}, these data satisfy the technical assumptions imposed in~\cite{abe}.)
\end{rmk}

We will now denote by $(\bt^*)^\wedge$ the spectrum of the completion of $\scO(\bt^*)$ with respect to the ideal $\bt \cdot \scO(\bt^*)$. 
We will consider a third family of categories as in~\S\ref{ss:Hecke-a-la-Abe}, now associated with the ring $\scO((\bt^*)^\wedge)$. To check that conditions~\eqref{eqn:assumption-deltas} hold in this case, one can e.g.~use the following ``additive'' variant of~\cite[Lemma~3.3]{reg-quotient-pt2} (applied to the Levi factor of $\bG$ associated with $s$, when $s \in \bSf$). Here we denote by $(\bt^*/\bWf)^\wedge$ the spectrum of the completion of $\scO(\bt^*/\bWf)$ with respect to the ideal corresponding to the image of $0 \in \bt^*$, we consider the fiber product
\[
\bt^* \times_{\bt^*/\bWf} \bt^*,
\]
and we denote by $(\bt^* \times_{\bt^*/\bWf} \bt^*)^\wedge$ the spectrum of the completion of $\scO(\bt^* \times_{\bt^*/\bWf} \bt^*)$ with respect to the maximal ideal corresponding to $(0,0) \in \bt^* \times_{\bt^*/\bWf} \bt^*$.

\begin{lem}
\phantomsection
\label{lem:Dw-fiber-prod-add}
\begin{enumerate}
\item
\label{it:Dw-fiber-prod-add-1}
There exist canonical isomorphisms of $\bk$-schemes
\[
(\bt^*)^\wedge \cong \bt^* \times_{\bt^*/\bWf} (\bt^*/\bWf)^\wedge
\]
and
\begin{multline*}
(\bt^* \times_{\bt^*/\bWf} \bt^*)^\wedge \cong (\bt^*)^\wedge \times_{\bt^*} (\bt^* \times_{\bt^*/\bWf} \bt^*) \cong (\bt^* \times_{\bt^*/\bWf} \bt^*) \times_{\bt^*} (\bt^*)^\wedge \\
\cong (\bt^* \times_{\bt^*/\bWf} \bt^*) \times_{\bt^* / \bWf} (\bt^*/\bWf)^\wedge \cong (\bt^*)^\wedge \times_{(\bt^*/\bWf)^\wedge} (\bt^*)^\wedge
\end{multline*}
where in the first, resp.~second, fiber product the morphism $\bt^* \times_{\bt^*/\bWf} \bt^* \to \bt^*$ is induced by projection on the first, resp.~second, factor. Moreover, $\scO((\bt^*)^\wedge)$ is finite and free (in particular, flat) over $\scO((\bt^*/\bWf)^\wedge)$.
\item
\label{it:Dw-fiber-prod-add-2}
The natural morphism $\scO((\bt^*/\bWf)^\wedge) \to \scO((\bt^*)^\wedge)^{\bWf}$ is an isomorphism.
\end{enumerate}
\end{lem}

\begin{proof}
The proof of~\eqref{it:Dw-fiber-prod-add-1} is similar to that of the corresponding claim in~\cite[Lemma~3.3]{reg-quotient-pt2}, replacing the reference to the Pittie-Steinberg theorem to a reference to the main result of~\cite{demazure} (applied to the ``precised'' root system $\mathfrak{R}^\vee$ in $\rmX_*(\bT)$; our assumptions guarantee that $p$ is not a torsion prime for this root system). 

 To prove~\eqref{it:Dw-fiber-prod-add-2}, let us set $\cK_\add := \bt \cdot \scO(\bt^*)$ and $\cJ_\add := \cK_\add \cap \scO(\bt^*/\bWf)$, so that $\scO((\bt^*)^\wedge)$ is the completion of $\scO(\bt^*)$ with respect to $\cK_\add$ and $\scO((\bt^*/\bWf)^\wedge)$ is the completion of $\scO(\bt^*/\bWf)$ with respect to $\cJ_\add$.
It is easily seen that $\scO((\bt^*)^\wedge)^{\bWf}$ is the completion of $\scO(\bt^*/\bWf)$ with respect to the (decreasing) family of ideals
 \[
 \bigl( (\cK_\add)^n \cap \scO(\bt^*/\bWf) : n \in \Z_{\geq 1} \bigr).
 \]
 Now for any $n \geq 1$ we have $(\cJ_\add)^n \subset (\cK_\add)^n \cap \scO(\bt^*/\bWf)$. On the other hand, as in the proof of~\cite[Lemma~3.3]{reg-quotient-pt2} there exists $N$ such that $(\cK_\add)^N \subset \cJ_\add \cdot \scO(\bt^*)$. We deduce that for any $n \geq 1$ we have $(\cK_\add)^{nN} \subset (\cJ_\add)^n \cdot \scO(\bt^*)$, and then since the embedding $\scO(\bt^*/\bWf) \hookrightarrow \scO(\bt^*)$ admits an $\scO(\bt^*/\bWf)$-linear retraction (again by the main result of~\cite{demazure}), we deduce that $(\cK_\add)^{nN} \cap \scO(\bt^*/\bWf) \subset (\cJ_\add)^n$, so that our two completions are isomorphic.
\end{proof}

The categories $\sfK(\scO((\bt^*)^\wedge))$ and $\BSK(\scO((\bt^*)^\wedge))$ will be denoted
\[
\sfK_\add^\wedge \quad \text{ and } \quad \BSK_\add^{\wedge}
\]
respectively.

\subsection{A technical lemma on representations of affine group schemes}

As a preparation for the next subsection, here we prove a technical lemma on (affine) group schemes and their categories of representations.

Given a commutative noetherian ring $R$ (or, equivalently, an affine noetherian scheme $X=\Spec(R)$) and a $R$-Hopf algebra $A$ flat over $R$ (or, equivalently, a flat affine group scheme $H=\Spec(A)$ over $X$), we will denote by $\Rep(H)$ the category of representations of $H$ on coherent $\scO_X$-modules, or in other words the category of $A$-comodules which are finitely generated as $R$-modules. 

If $R$ is endowed with a $\Z$-grading (equivalently, $X$ is equipped with an action of the multiplicative group over $\Z$) and $A$ with a compatible $\Z$-grading (equivalently, $H$ is endowed with an action of the multiplicative group such that the structure morphism $H \to X$, the multiplication morphism $H \times_X H \to H$, the inversion morphism $H \to H$ and the unit section $X \to H$ are equivariant), we will denote by $\Rep^{\Gm}(H)$ the category of equivariant (for the multiplicative group) representations of $H$ on coherent $\scO_X$-modules, or in other words the category of $\Z$-graded $A$-comodules which are finitely generated as $R$-modules. (We will refer to this setting as the ``graded setting.'')
This category admits a ``shift of grading'' functor $(1)$, defined with the same convention as in~\S\ref{ss:graded-Hecke}. 
We have a canonical forgetful functor
\[
\For^{\Gm} : \Rep^{\Gm}(H) \to \Rep(H)
\]
which satisfies $\For^{\Gm} \circ (1) = \For^{\Gm}$.

%

If $R$ and $A$ are as above, with $X=\Spec(R)$, $H=\Spec(A)$, given a commutative noetherian right $R'$ and a ring morphism $R \to R'$, we can set $X':=\Spec(R')$ and consider the group scheme obtained by base change
\[
X' \times_X H = \Spec(R' \otimes_R A)
\]
and its category of representations (finite over $R'$) $\Rep(X' \times_X H)$. We then have a canonical functor
\[
R' \otimes_R (-) : \Rep(H) \to \Rep(X' \times_X H).
\]

\begin{lem}
\label{lem:compatibilities-gp-schemes}
Let $R$ and $A$ be as above, and set $X=\Spec(R)$, $H=\Spec(A)$.
\begin{enumerate}
\item
\label{it:For-Gm-gp-schemes}
Assume we are in the graded setting.
For any $M,M'$ in $\Rep^{\Gm}(H)$, the functor $\For^{\Gm}$ induces an isomorphism
\[
\bigoplus_{n \in \Z} \Hom_{\Rep^{\Gm}(H)}(M,M'(n)) \simto \Hom_{\Rep(H)}(\For^{\Gm}(M), \For^{\Gm}(M')).
\]
\item
\label{it:base-change-gp-schemes}
Assume we are given a commutative noetherian ring $R'$ and a flat morphism $R \to R'$. Then
for any $M,M'$ in $\Rep(H)$, the functor $R' \otimes_R (-)$ induces an isomorphism
\[
R' \otimes_R \Hom_{\Rep(H)}(M,M') \simto \Hom_{\Rep(X' \times_X H)}(R' \otimes_R M, R' \otimes_R M').
\]
\end{enumerate}
\end{lem}

\begin{proof}
\eqref{it:For-Gm-gp-schemes}
We will prove this property when $M'$ is more generally a $\Z$-graded $A$-comodule which is not necessarily finitely generated over $R$. (The category of such objects will be denoted $\Rep^{\Gm}_\infty(H)$.) First, assume that $M' = V \otimes_R A$ for some graded $R$-module $V$ (with the coaction induced by the comultiplication in $A$). Then by Frobenius reciprocity (\cite[Proposition~I.3.4]{jantzen}), for any $n \in \Z$ we have
\[
\Hom_{\Rep_\infty^{\Gm}(H)}(M,M'(n)) \cong \Hom_{\Mod^{\Z}(R)}(M,V(n))
\]
where $\Mod^{\Z}(R)$ is the category of $\Z$-graded $R$-modules, and
\[
\Hom_{\Rep(H)}(\For^{\Gm}(M), \For^{\Gm}(M')) \cong \Hom_{\Mod(R)}(M,V).
\]
Now it is a classical fact that since $M$ is finitely generated over $R$ the forgetful functor induces an isomorphism
\[
\bigoplus_{n \in \Z} \Hom_{\Mod^{\Z}(R)}(M,V(n)) \simto \Hom_{\Mod(R)}(M,V);
\]
the desired claim follows in this case.

The case of a general $\Z$-graded $A$-comodule $M'$ follows from this special case using the five lemma and the fact that for any such $M'$ the coaction defines an injective morphism of $\Z$-graded $A$-comodules $M' \to M' \otimes_R A$, where in the right-hand side $M'$ is regarded as a graded $R$-module.

\eqref{it:base-change-gp-schemes}
As explained e.g.~in~\cite[Lemma~3.8(2)]{brHecke}, the $R$-module $\Hom_R(M,M')$ admits a natural structure of $A$-comodule, and we have
\[
\Hom_{\Rep(H)}(M,M') = \bigl( \Hom_{R}(M,M') \bigr)^H,
\]
where $(-)^H$ is the functor of $H$-fixed points. Similarly, $\Hom_{R'}(R' \otimes_R M, R' \otimes_R M')$ admits a natural structure of $(R' \otimes_R A)$-comodule, and we have
\[
\Hom_{\Rep(X' \times_X H)}(R' \otimes_R M, R' \otimes_R M') = \bigl( \Hom_{R'}(R' \otimes_R M, R' \otimes_R M')\bigr)^{X' \times_X H}
\]
where $(-)^{X' \times_X H}$ is the functor of $(X' \times_X H)$-fixed points. Now we have
\[
\Hom_{R'}(R' \otimes_R M, R' \otimes_R M') = \Hom_{R}(M, R' \otimes_R M'),
\]
and since $R'$ is flat over $R$ we have
\[
\Hom_{R}(M, R' \otimes_R M') = R' \otimes_R \Hom_{R}(M, M')
\]
by~\cite[Lemma~3.8(1)]{brHecke}. Finally by~\cite[Equation~I.2.10(3)]{jantzen}, using again our flatness assumption we have
\[
\bigl( R' \otimes_R \Hom_{R}(M, M') \bigr)^{X' \times_X H} = R' \otimes_R \bigl( \Hom_{R}(M,M') \bigr)^H.
\]
Combining these isomorphisms we deduce the desired claim.
\end{proof}

\subsection{Additive Hecke categories and representations of the (additive) universal centralizer}

From now on we
fix a Kostant section $\bS \subset \bg$ as in~\cite[\S 2.3]{brHecke}, and denote by $\bS^*$ its image under our identification $\bg \simto \bg^*$. Here, a Kostant section is an ``additive'' variant of the Steinberg section $\mathbf{\Sigma}$; what we will use is that $\bS$ is an affine subspace of $\bg$, contained in the open subset of regular elements, and that the composition $\bS \hookrightarrow \bg \to \bg/\bG \cong \bt / \bWf$ is an isomorphism. In particular, from this fact we deduce that the coadjoint quotient provides an isomorphism $\bS^* \simto \bt^*/\bWf$.

The same considerations as in~\S\ref{ss:representations-Iadj} lead to the definition of the universal centralizer group scheme $\bbJ_{\bS^*}$ over $\bS^*$, see~\cite[\S 2.3]{brHecke} for details. This is a smooth affine group scheme over $\bS^*$, endowed with an action of $\Gm$ which is compatible (in the sense considered above Lemma~\ref{lem:compatibilities-gp-schemes}) with the action on $\bt^*/\bWf$ such that the quotient morphism $\bt^* \to \bt^*/\bWf$ is $\Gm$-equivariant, where $t \in \Gm$ acts on $\bt^*$ by multiplication by $t^{-2}$. We can then consider the group scheme
\[
\bbI_{\bS^*} := (\bt^* \times_{\bt^*/\bWf} \bt^*) \times_{\bt^*/\bWf} \bbI_{\bS^*}
\]
and the associated category
\[
\Rep^{\Gm}(\bbI_{\bS^*})
\]
of $\Gm$-equivariant representations on coherent sheaves. This category admits a canonical convolution product defining a monoidal structure. If we denote by
\[
\Rep^{\Gm}_{\mathrm{fl}}(\bbI_{\bS^*})
\]
the full subcategory whose objects are the representations whose underlying coherent sheaves are flat with respect to the second projection $\bt^* \times_{\bt^*/\bWf} \bt^* \to \bt^*$, then this full subcategory is stable under convolution, hence a monoidal category.

On the other hand, set
\[
\bbI_{\bS^*}^\wedge := (\bt^* \times_{\bt^*/\bWf} \bt^*)^\wedge \times_{\bt^* \times_{\bt^*/\bWf} \bt^*} \bbI_{\bS^*}.
\]
Then, once again, the abelian category $\Rep(\bbI_{\bS^*}^\wedge)$ of representations of $\bbI_{\bS^*}^\wedge$ on coherent $\scO_{(\bt^* \times_{\bt^*/\bWf} \bt^*)^\wedge}$-modules admits a canonical convolution product which makes it a monoidal category. If we denote by $\Rep_{\mathrm{fl}}(\bbI_{\bS^*}^\wedge)$ the full subcategory whose objects are the representations whose underlying coherent sheaves are flat with respect to the second projection $(\bt^* \times_{\bt^*/\bWf} \bt^*)^\wedge \to (\bt^*)^\wedge$, then this subcategory is stable under convolution, hence a monoidal category. 

\begin{prop}
\phantomsection
\label{prop:Abe-univ-cent-add}
\begin{enumerate}
\item
\label{eqn:Abe-univ-cent-add-1}
There exists a canonical fully faithful monoidal functor
\begin{equation}
\label{eqn:functor-Abe-Dadd}
\Rep^{\Gm}_{\mathrm{fl}}(\bbI_{\bS^*}) \to \sfK_\add,
\end{equation}
whose essential image contains the objects $B_s$ ($s \in \bSaff$) and $F_w$ ($w \in \bW$).
\item
\label{eqn:Abe-univ-cent-add-2}
There exists a canonical fully faithful monoidal functor
\begin{equation}
\label{eqn:functor-Abe-Daddwedge}
\Rep_{\mathrm{fl}}(\bbI_{\bS^*}^\wedge) \to \sfK_\add^\wedge,
\end{equation}
whose essential image contains the objects $B_s$ ($s \in \bSaff$) and $F_w$ ($w \in \bW$).
\end{enumerate}
\end{prop}

\begin{proof}
\eqref{eqn:Abe-univ-cent-add-1}
This statement is proved in~\cite[Proposition~2.7 and Lemma~2.9]{brHecke}.

\eqref{eqn:Abe-univ-cent-add-2}
The proof is similar to that of~\cite[Proposition~3.9]{reg-quotient-pt2}.
\end{proof}

More specifically, one can define \emph{canonical} objects in the category $\Rep^{\Gm}_{\mathrm{fl}}(\bbI_{\bS^*})$, resp.~in $\Rep_{\mathrm{fl}}(\bbI_{\bS^*}^\wedge)$, whose image under~\eqref{eqn:functor-Abe-Dadd}, resp.~\eqref{eqn:functor-Abe-Daddwedge}, are the corresponding objects $B_s$ and $F_w$. Using these objects one obtains that the functors~\eqref{eqn:functor-HBS-K} in these two settings factor through (fully faithful) monoidal functors
\[
\BSK_\add \to \Rep^{\Gm}_{\mathrm{fl}}(\bbI_{\bS^*}), \qquad
\BSK_\add^\wedge \to \Rep_{\mathrm{fl}}(\bbI_{\bS^*}^\wedge).
\]
Using the second of these functors one can define a category $\BSRep(\bbI_{\bS^*}^\wedge)$ of ``Bott--Samelson type'' representations of $\bbI_{\bS^*}^\wedge$, with objects the collections $(\omega, s_1, \dots, s_i)$ with $\omega \in \mathbf{\Omega}$ and $s_1, \dots, s_i \in \bSaff$, and which is canonically equivalent to $\BSK_\add^\wedge$. One can also define the category $\SRep(\bbI_{\bS^*}^\wedge)$ of ``Soergel type'' representations as the karoubian closure of the additive hull of $\BSRep(\bbI_{\bS^*}^\wedge)$; equivalently, this category identifies with the full subcategory of $\Rep_{\mathrm{fl}}(\bbI_{\bS^*}^\wedge)$ whose objects are direct sums of direct summands of objects in the image of $\BSK_\add^\wedge$.

We deduce from Proposition~\ref{prop:Abe-univ-cent-add} the following property. (We expect this proposition to admit a direct algebraic proof, but the proof given below relies on geometry and Proposition~\ref{prop:Abe-univ-cent-add}.)

 \begin{prop}
\label{prop:relation-Hecke-cat-add}
There exists a monoidal functor
 \[
  \mathsf{F} : \sfK_\add \to \sfK_\add^{\wedge}
 \]
 which satisfies $\mathsf{F} \circ (1) = \mathsf{F}$ and sends each object $F_w$ ($w \in \bW$) and $B_s$ ($s \in \bSaff$) in $\sfK_\add$ to the corresponding object in $\sfK_\add^\wedge$, and such that $\mathsf{F}$ induces an isomorphism
\[
  \Hom^\bullet_{\BSK_\add}(M,M') \otimes_{\scO(\bt^*)} \scO((\bt^*)^\wedge) \simto \Hom_{\BSK_\add^\wedge}(\mathsf{F}(M),\mathsf{F}(M'))
\]
for any
$M,M'$ in $\BSK_\add$.
\end{prop}

\begin{proof}
The functor $\mathsf{F}$ is defined by
\[
\mathsf{F}(M) = M \otimes_{\scO(\bt^*)} \scO((\bt^*)^\wedge),
\]
where in the right-hand side we omit the functor forgetting the $\Z$-grading. Let us first explain why this indeed defines a functor from $\sfK_\add$ to $\sfK_\add^{\wedge}$. Here since $M$ is an object in $\sfK_\add$, it admits in particular an action of $\scO(\bt^* \times_{\bt^*/\bWf} \bt^*)$. Hence $M \otimes_{\scO(\bt^*)} \scO((\bt^*)^\wedge)$ admits an action of
\[
\scO(\bt^* \times_{\bt^*/\bWf} \bt^*) \otimes_{\scO(\bt^*)} \scO((\bt^*)^\wedge),
\]
which identifies with $\scO((\bt^* \times_{\bt^*/\bWf} \bt^*)^\wedge)$ by Lemma~\ref{lem:Dw-fiber-prod-add}. This object can therefore be regarded as a (finitely generated) $\scO((\bt^*)^\wedge)$-bimodule. On the other hand, we have
\begin{multline*}
\bigl( M \otimes_{\scO(\bt^*)} \scO((\bt^*)^\wedge) \bigr) \otimes_{\scO((\bt^*)^\wedge)} \mathrm{Frac}(\scO((\bt^*)^\wedge)) \cong \\
 \bigl( M \otimes_{\scO(\bt^*)} \mathrm{Frac}(\scO(\bt^*)) \bigr) \otimes_{\mathrm{Frac}(\scO(\bt^*))} \mathrm{Frac}(\scO((\bt^*)^\wedge)).
\end{multline*}
Here we are given a decomposition of $M \otimes_{\scO(\bt^*)} \mathrm{Frac}(\scO(\bt^*))$ parametrized by $\bW$, which induces a decomposition of $\bigl( M \otimes_{\scO(\bt^*)} \scO((\bt^*)^\wedge) \bigr) \otimes_{\scO((\bt^*)^\wedge)} \mathrm{Frac}(\scO((\bt^*)^\wedge))$ parametrized by $\bW$. Finally $M \otimes_{\scO(\bt^*)} \scO((\bt^*)^\wedge)$ is flat over $\scO((\bt^*)^\wedge)$ for the action on the right, hence it indeed admits a canonical structure of object in $\sfK_\add^{\wedge}$.

It is easily checked that $\mathsf{F}$ has a canonical monoidal structure, and the required action on the objects $F_w$ and $B_s$.
To check that this functor has the required property on morphism spaces,
we consider the equivalences of Proposition~\ref{prop:Abe-univ-cent-add}, and the functor of pullback under the natural morphism $(\bt^* \times_{\bt^*/\bWf} \bt^*)^\wedge \to \bt^* \times_{\bt^*/\bWf} \bt^*$ (and forgetting the grading). This defines a natural monoidal functor
\begin{equation}
\label{eqn:functor-Dadd-Daddwedge}
\Rep^{\Gm}_{\mathrm{fl}}(\bbI_{\bS^*}) \to \Rep_{\mathrm{fl}}(\bbI_{\bS^*}^\wedge)
\end{equation}
and, in view of the identification $\scO((\bt^* \times_{\bt^*/\bWf} \bt^*)^\wedge) \cong \scO(\bt^* \times_{\bt^*/\bWf} \bt^*) \otimes_{\scO(\bt^*)} \scO((\bt^*)^\wedge)$, the diagram
\[
\xymatrix@C=2cm{
\Rep^{\Gm}_{\mathrm{fl}}(\bbI_{\bS^*}) \ar[d]_-{\eqref{eqn:functor-Dadd-Daddwedge}} \ar[r]^-{\eqref{eqn:functor-Abe-Dadd}} & \sfK_\add \ar[d]^-{\mathsf{F}} \\
\Rep_{\mathrm{fl}}(\bbI_{\bS^*}^\wedge) \ar[r]^-{\eqref{eqn:functor-Abe-Daddwedge}} & \sfK_\add^\wedge
}
\]
commutes. The desired property of $\mathsf{F}$ therefore follows from the corresponding property of the functor~\eqref{eqn:functor-Dadd-Daddwedge}, which itself follows from Lemma~\ref{lem:compatibilities-gp-schemes} and the identification $\scO((\bt^* \times_{\bt^*/\bWf} \bt^*)^\wedge) \cong \scO(\bt^* \times_{\bt^*/\bWf} \bt^*) \otimes_{\scO(\bt^*)} \scO((\bt^*)^\wedge)$.
\end{proof}

\subsection{Relation between the ``multiplicative'' and ``additive'' Hecke categories}
\label{ss:relation-mult-add}

Finally we explain the relation between the (completed) ``additive'' and ``multiplicative'' Hecke categories.

\begin{lem}
\label{lem:Abe-add-mul}
Assume that there exists an \'etale (in particular, central) isogeny $\bG' \to \bG$ and a morphism $\bG' \to \mathrm{Lie}(\bG')$ which is $\bG'$-equivariant (for the adjoint actions), sends $e$ to $0$, and is \'etale at $e$.
Then there exists a $\bWf$-equivariant isomorphism $\scO((\bt^*)^\wedge) \simto \scO(\bT^\wedge)$, from which we obtain an equivalence of monoidal categories 
\[
\sfK_\add^{\wedge} \cong \sfK^{\wedge} 
\]
sending each object $F_w$ ($w \in \bW$) or $B_s$ ($s \in \bSaff$) in $\sfK_\add^{\wedge}$ to the corresponding object in $\sfK^{\wedge}$. As a consequence, we obtain an equivalence of monoidal categories
\[
\BSK_\add^{\wedge} \cong \BSK^{\wedge}
\]
which is the identity on objects.
\end{lem}

\begin{proof}
By assumption there exists a $\bG$-equivariant isomorphism $\bg \simto \bg^*$; if one identifies $\bt^*$ with the subspace of $\bg^*$ consisting of linear forms vanishing on each root subspace, then this isomorphism must restrict to an isomorphism from $\bt=\bg^{\bT}$ to $\bt^* = (\bg^*)^\bT$, which is $\bWf$-invariant. To construct our isomorphism it therefore suffices to construct a $\bWf$-equivariant isomorphism from $\scO(\bT^\wedge)$ to the completion of $\scO(\bt)$ with respect to the ideal corresponding to $0$.

Consider now an isogeny $\bG' \to \bG$ as in the statement.
If $\bT' \subset \bG'$ is the preimage of $\bT$, then $\bT'$ is a maximal torus in $\bG'$, and our isogeny restricts to an \'etale morphism $\bT' \to \bT$ sending $e$ to $e$. It therefore induces an isomorphism between $\scO(\bT^\wedge)$ and the completion $\scO(\bT')^\wedge$ of $\scO(\bT')$ with respect to the ideal corresponding to $e$, and also an isomorphism from $\mathrm{Lie}(\bT')$ to $\bt$. The Weyl group of $(\bG',\bT')$ canonically identifies with $\bWf$, and both of our isomorphisms are $\bWf$-equivariant.

Our morphism $\bG' \to \mathrm{Lie}(\bG')$ must restrict to a $\bWf$-equivariant morphism from $\bT' = (\bG')^{\bT'}$ to $\mathrm{Lie}(\bT') = (\mathrm{Lie}(\bG'))^{\bT'}$. Moreover, this morphism sends $e$ to $0$ and is \'etale at $e$ (e.g.~by consideration of tangent spaces). It therefore induces a $\bWf$-equivariant isomorphism between $\scO(\bT')^\wedge$ and the completion of $\scO(\mathrm{Lie}(\bT'))$ with respect to the ideal corresponding to $0$. Combining these isomorphisms we deduce the desired isomorphism
\[
\scO((\bt^*)^\wedge) \simto \scO(\bT^\wedge).
\]
Once this isomorphism is constructed, we deduce the desired equivalences using Lemma~\ref{lem:isom-equiv-Abe-cat}.
 \end{proof}
 
 From Lemma~\ref{lem:Abe-add-mul}, together with the equivalence~\eqref{eqn:equiv-BSK-BSRep} and its analogue deduced from Proposition~\ref{prop:Abe-univ-cent-add}\eqref{eqn:Abe-univ-cent-add-2}, we obtain an equivalence of additive monoidal categories
 \begin{equation}
  \SRep(\bbI_{\mathbf{\Sigma}}^\wedge) \simto \SRep(\bbI_{\bS^*}^\wedge).
 \end{equation}
 
 \begin{rmk}
 \label{rmk:assumptions-Hecke}
The assumption in Lemma~\ref{lem:Abe-add-mul} holds at least in the following cases:
\begin{enumerate}
\item
$\bG=\mathrm{GL}_n(\bk)$;
\item
$p$ is very good.
\end{enumerate}
In fact, in the first case one can take $\bG'=\bG$, with the morphism $\mathrm{GL}_n(\bk) \to \mathfrak{gl}_n(\bk)$ given by $X \mapsto X-\mathrm{I}_n$. For the second case one observes first that if $p$ is very good and $\bG$ is semisimple (and simply connected) then there exists a morphism $\bG \to \bg$ sending $e$ to $0$ and \'etale at $e$; see~\cite[Remark~8.1]{reg-quotient-pt3} for details. The similar claim of course also holds if $\bG$ is a torus. Finally, for a general $\bG$, as explained in~\cite[\S 1.18]{jantzen} there exists a torus $\mathbf{H}$ and an isogeny $\mathscr{D}(\bG) \times \mathbf{H} \to \bG$ (where $\mathscr{D}(\bG)$ is the derived subgroup of $\bG$) whose kernel is a subgroup of the center of $\mathscr{D}(\bG)$. Since $p$ is very good this center is a discrete group, hence this kernel is smooth, proving that the isogeny is \'etale. One can therefore take $\bG'=\mathscr{D}(\bG) \times \mathbf{H}$.
\end{rmk}

\section{Tilting perverse perverse on affine flag varieties}
\label{sec:tilting}


\subsection{Sheaves on affine flag varieties}
\label{ss:sheaves-Fl}

Now we fix an algebraically closed field $\F$ of positive characteristic and a connected reductive algebraic group $G$ over $\F$. We also choose a Borel subgroup $B \subset G$ and a maximal torus $T \subset B$. To $G$ one can associate its loop group $\Loop G$ (a group ind-scheme over $\F$) and arc group $\Loop^+ G$ (a group scheme over $\F$, not of finite type unless $G$ is the trivial group). For definitions, see e.g.~\cite[\S 4.1]{reg-quotient-pt2}. The associated Iwahori subgroup $\Iw \subset \Loop^+ G$ is the preimage of $B$ under the canonical morphism $\Loop^+ G \to G$. The prounipotent radical $\Iwu$ of $\Iw$ is then the preimage of the unipotent radical $U$ of $B$ under this map. We consider the affine flag variety $\Fl_G$, namely the ind-scheme over $\F$ defined as
\[
 \Fl_G = (\Loop G / \Iw)_{\mathrm{fppf}}.
\]
We also have a canonical $T$-torsor over this ind-scheme, defined as
\[
 \tFl_G = (\Loop G / \Iwu)_{\mathrm{fppf}}.
\]
The map realizing this torsor will be denoted $\pi : \tFl_G \to \Fl_G$. Finally, the affine Grassmannian $\Gr_G$ is the ind-scheme defined as
\[
 \Gr_G = (\Loop G / \Loop^+ G)_{\mathrm{fppf}}.
\]

We now choose a prime number $p$ invertible in $\F$ and denote by $\bk$ an algebraic closure of the finite field $\F_p$.
We will consider the following categories:
\begin{itemize}
\item the (\'etale) $\Loop^+ G$-equivariant derived category of $\bk$-sheaves on $\Gr_G$, denoted $\sfD_{\Loop^+ G,\Loop^+ G}$;
 \item the (\'etale) $\Iw$-equivariant derived category of $\bk$-sheaves on $\Fl_G$, denoted $\sfD_{\Iw,\Iw}$;
 \item the (\'etale) $\Iwu$-equivariant derived category of $\bk$-sheaves on $\Fl_G$, denoted $\sfD_{\Iwu,\Iw}$;
 \item the (\'etale) completed $\Iwu$-equivariant derived category of $\bk$-sheaves on $\tFl_G$  which are constructible with respect to the stratification by $\Iw$-orbits, denoted $\sfD^\wedge_{\Iwu,\Iwu}$.
\end{itemize}
Here the fourth case relies on constructions due to Yun in an appendix to~\cite{by}; see also~\cite{bezr} or~\cite[\S 6.1]{reg-quotient-pt2}. The structures on these categories that will be used below are the following.
\begin{enumerate}
\item 
 The category $\sfD_{\Loop^+ G,\Loop^+ G}$ admits a canonical monoidal structure, with product given by the convolution bifunctor, denoted $\star_{\Loop^+ G}$, see e.g.~\cite[\S 4.3]{reg-quotient-pt2}.
 \item 
 The category $\sfD_{\Iw,\Iw}$ admits a canonical monoidal structure, with product given by the convolution bifunctor, denoted $\star_{\Iw}$, see e.g.~\cite[\S 4.2]{reg-quotient-pt2}.
 \item
 The category $\sfD^\wedge_{\Iwu,\Iwu}$ admits a canonical monoidal structure, with product given by the convolution bifunctor, denoted $\hatstar$, see~\cite[\S 6.1]{reg-quotient-pt2}.
 \item
 The category $\sfD_{\Iwu,\Iw}$ admits a canonical left action of $\sfD^\wedge_{\Iwu,\Iwu}$, and a canonical commuting right action of $\sfD_{\Iw,\Iw}$. These actions are given by convolution, and the corresponding bifunctors will also be denoted $\hatstar$ and $\star_{\Iw}$.
 \item
 We have a ``forgetful'' functor $\For^{\Iw}_{\Iwu} : \sfD_{\Iw,\Iw} \to \sfD_{\Iwu,\Iw}$.
 \item
 There exists a canonical functor $\pi_\dag : \sfD^\wedge_{\Iwu,\Iwu} \to \sfD_{\Iwu,\Iw}$; see again~\cite[\S 6.1]{reg-quotient-pt2}.
 \item
 Let $T^\vee_\bk$ be the $\bk$-torus whose lattice of characters is $\rmX_*(T)$, and let $(T^\vee_\bk)^\wedge$ be the spectrum of the completion of the ring $\scO(T^\vee_\bk)$ with respect to the maximal ideal corresponding to the unit element. Then the category $\sfD^\wedge_{\Iwu,\Iwu}$ is naturally enriched in right modules over $\scO((T^\vee_\bk)^\wedge)$, via right monodromy. Moreover, for $\scF$, $\scG$ in $\sfD^\wedge_{\Iwu,\Iwu}$, the morphism induced by $\pi_\dag$ factors through a morphism
 \begin{equation}
 \label{eqn:Hom-pidag}
  \Hom_{\sfD^\wedge_{\Iwu,\Iwu}}(\scF,\scG) \otimes_{\scO((T^\vee_\bk)^\wedge)} \bk \to \Hom_{\sfD_{\Iwu,\Iw}}(\pi_\dag \scF, \pi_\dag \scG),
 \end{equation}
 where $\bk$ is seen as the quotient of $\scO((T^\vee_\bk)^\wedge)$ by its unique maximal ideal.
 \end{enumerate}

Each of the categories $\sfD_{\Loop^+ G,\Loop^+ G}$, $\sfD_{\Iw,\Iw}$, $\sfD_{\Iwu,\Iw}$, $\sfD^\wedge_{\Iwu,\Iwu}$ admits a ``perverse'' t-struc\-ture, such that the functor $\For^{\Iw}_{\Iwu}$ is t-exact; the corresponding hearts will be denoted $\sfP_{\Loop^+ G,\Loop^+ G}$, $\sfP_{\Iw,\Iw}$, $\sfP_{\Iwu,\Iw}$, $\sfP^\wedge_{\Iwu,\Iwu}$. It is a standard fact (see~\cite{mv} for the original reference, and~\cite{br} for a survey) that the subcategory $\sfP_{\Loop^+ G,\Loop^+ G}$ is stable under the convolution product $\star_{\Loop^+ G}$. Moreover, if we denote by $G^\vee_\bk$ the connected reductive algebraic group over $\bk$ which is Langlands dual to $G$, then there exists a canonical equivalence of monoidal categories
\begin{equation}
\label{eqn:Satake-equiv}
 (\sfP_{\Loop^+ G,\Loop^+ G}, \star_{\Loop^+ G}) \cong (\Rep(G^\vee_\bk), \otimes)
\end{equation}
where the right-hand side is the category of finite-dimensional algebraic representations of $G^\vee_\bk$. 
By definition of the Langlands dual group, $T^\vee_\bk$ identifies canonically with a maximal torus in $G^\vee_\bk$. We will denote by $B^\vee_\bk \subset G^\vee_\bk$ the Borel subgroup of $G^\vee_\bk$ whose roots are the coroots associated with the roots of $B$.

Let $\Wf$ be the Weyl group of $(G,T)$, and consider the extended affine Weyl group
\[
 W = \Wf \ltimes \rmX_*(T).
\]
The subgroup $\Waff$ given by the semidirect product of $\Wf$ with the coroot lattice is a normal subgroup.
The Bruhat decomposition provides a natural parametrization of the $\Iw$-orbits in $\Fl_G$ or $\tFl_G$ by $W$. In particular, each $\Iw$-orbit on $\Fl_G$ is isomorphic to an affine space, and the dimension of the orbit labelled by $w$ will be denoted $\ell(w)$. (It is a standard fact that $\Waff$ admits a canonical subset of Coxeter generators, for which the restriction of $\ell$ is the associated length function.)

To each $w \in W$ one can associate ``standard objects''
\[
 \Delta^{\Iw}_w \in \sfD_{\Iw,\Iw}, \qquad \Delta^\wedge_w \in \sfD^\wedge_{\Iwu,\Iwu}
\]
defined by taking the $!$-extension of a shift of an appropriate local system (or pro-local system) on the orbit labelled by $w$, see~\cite[\S 4.2 and~\S 6.2]{reg-quotient-pt2}, and ``costandard objects''
\[
 \nabla^{\Iw}_w \in \sfD_{\Iw,\Iw}, \qquad \nabla^\wedge_w \in \sfD^\wedge_{\Iwu,\Iwu}
\]
obtained by replacing $!$-extension by $*$-extension in this construction. In both cases these objects are perverse, and they satisfy
\[
 \pi_\dag(\Delta^\wedge_w) \cong \For^{\Iw}_{\Iwu}(\Delta^{\Iw}_w), \qquad 
 \pi_\dag(\nabla^\wedge_w) \cong \For^{\Iw}_{\Iwu}(\nabla^{\Iw}_w).
\]
For any $w \in W$ there exists a unique (up to scalar) nonzero morphism $\Delta_w^{\Iw} \to \nabla^{\Iw}_w$; its image is simple, and denoted $\IC_w$. (This is the intersection cohomology complex associated with the constant local system on the orbit labelled by $w$.)

For $w \in W$,
we will also set
\[
 \Delta_w = \For^{\Iw}_{\Iwu}(\Delta^{\Iw}_w), \quad \nabla_w = \For^{\Iw}_{\Iwu}(\nabla^{\Iw}_w).
\]
If $y,w \in W$, we will write $y \leq w$ if the $\Iw$-orbit in $\Fl_G$ labelled by $y$ is contained in the closure of the orbit labelled by $w$.

\begin{rmk}
 The objects $\Delta^{\Iw}_w$ and $\nabla^{\Iw}_w$ are canonical, and do not depend on any choice. The objects $\Delta^\wedge_w$ and $\nabla^\wedge_w$, however, are defined only up to isomorphism in general, since their construction depends on certain choices.
\end{rmk}

\subsection{Tilting perverse sheaves}
\label{ss:tilt-perv}

Recall that an object $\scF \in \sfP_{\Iwu,\Iw}$ is called \emph{tilting} if it admits a filtration with subquotients of the form $\Delta_w$ ($w \in W$) and a filtration with subquotients of the form $\nabla_w$ ($w \in W$). In this case the number of subquotients isomorphic to a given $\Delta_w$ in a filtration with subquotients of the form $\Delta_y$ ($y \in W$) is independent of the choice of filtration, and denoted
\[
 (\scF : \Delta_w).
\]
The general theory of highest weight categories guarantees that the following properties hold.
\begin{enumerate}
 \item 
 The full subcategory $\sfT_{\Iwu,\Iw}$ of $\sfP_{\Iwu,\Iw}$ whose objects are the tilting perverse sheaves is stable under direct sums and direct summands, and it satisfies the Krull--Schmidt property.
 \item
 For any $w \in W$ there exists a unique (up to isomorphism) indecomposable tilting perverse sheaf $\scT_w$ such that $(\scT_w : \Delta_w)=1$ and
 \[
  (\scT_w : \Delta_y) \neq 0 \quad \Rightarrow \quad y \leq w.
 \]
\item
The assignment $w \mapsto \scT_w$ induces a bijection between $W$ and the set of isomorphism classes of indecomposable tilting perverse sheaves in $\sfP_{\Iwu,\Iw}$.
\end{enumerate}

It is clear that the image under Verdier duality of a tilting perverse sheaf is again tilting. From this, it is not difficult to deduce that each $\scT_w$ is Verdier self-dual, and then that any tilting perverse sheaf is isomorphic to its image under Verdier duality. We deduce that if $\scF \in \sfT_{\Iwu,\Iw}$ the number of occurrences of a given $\nabla_w$ in a filtration of $\scF$ with subquotients of the form $\nabla_y$ ($y \in W$) is equal to $(\scF : \Delta_w)$. The following property follows.
\begin{enumerate}
\setcounter{enumi}{3}
 \item 
 \label{it:dim-Hom-tilt}
 For any $\scF,\scG$ in $\sfT_{\Iwu,\Iw}$ we have
 \[
  \dim \Hom_{\sfT_{\Iwu,\Iw}}(\scF,\scG) = \sum_{w \in W} (\scF : \Delta_w) \cdot (\scG : \Delta_w).
 \]
 \end{enumerate}

Similarly, an object $\scF \in \sfP^\wedge_{\Iwu,\Iwu}$ is called \emph{tilting} if it admits a filtration with subquotients of the form $\Delta^{\wedge}_w$ ($w \in W$) and a filtration with subquotients of the form $\nabla^{\wedge}_w$ ($w \in W$). In this case the number of subquotients isomorphic to $\Delta^{\wedge}_w$ in such a filtration is independent of the choice of filtration, and denoted
\[
 (\scF : \Delta^{\wedge}_w).
\]
From the definition we see that the functor $\pi_\dag$ sends tilting perverse sheaves in $\sfP^\wedge_{\Iwu,\Iwu}$ to tilting perverse sheaves in $\sfP_{\Iwu,\Iw}$, and that we have
\begin{equation}
\label{eqn:mult-tilt-wedge}
 (\scF : \Delta^{\wedge}_w) = (\pi_\dag \scF : \Delta_w)
\end{equation}
for any $\scF$ tilting in $\sfP^\wedge_{\Iwu,\Iwu}$ and $w \in W$.
In fact, it turns out that an object $\scF \in \sfD^\wedge_{\Iwu,\Iwu}$ is a tilting perverse sheaf iff $\pi_\dag(\scF)$ is a tilting perverse sheaf. As above the following properties hold.
\begin{enumerate}
\setcounter{enumi}{4}
 \item 
 The full subcategory $\sfT^\wedge_{\Iwu,\Iwu}$ of $\sfP^\wedge_{\Iwu,\Iwu}$ whose objects are the tilting perverse sheaves is stable under direct sums and direct summands, and it satisfies the Krull--Schmidt property. It is also stable under the monoidal product $\hatstar$.
 \item
 For any $w \in W$ there exists a unique (up to isomorphism) object $\scT^\wedge_w \in \sfD^\wedge_{\Iwu,\Iwu}$ such that $\pi_\dag(\scT^\wedge_w) \cong \scT_w$; this object is an indecomposable tilting perverse sheaf.
\item
The assignment $w \mapsto \scT^\wedge_w$ induces a bijection between $W$ and the set of isomorphism classes of indecomposable tilting perverse sheaves in $\sfP^\wedge_{\Iwu,\Iwu}$.
\item
\label{it:tilting-pidag}
For any $\scF$, $\scG$ in $\sfT^\wedge_{\Iwu,\Iwu}$, the morphism~\eqref{eqn:Hom-pidag} is an isomorphism.
\end{enumerate}
For details and references on all of this, see~\cite[\S 6.3]{reg-quotient-pt2}.

For $w \in W$ such that $\ell(w)=0$ we have $\scT^\wedge_w \cong \Delta^\wedge_w \cong \nabla^\wedge_w$. If $w \in \Wf$ satisfies $\ell(w)=1$, the object $\scT^\wedge_s$ also admits an explicit construction, see~\cite[\S 6.6]{reg-quotient-pt2}. For general $w$ there is no such description, and in fact no canonical representative for $\scT^\wedge_w$ (nor for $\scT_w$).

\subsection{Relation with the Hecke category}
\label{ss:relation-Hecke}

From now on we make the following assumptions:
\begin{enumerate}
 \item 
 the quotient of $\rmX^*(T)$ by the root lattice of $(G,T)$ is free;
 \item
 the quotient of $\rmX_*(T)$ by the coroot lattice of $(G,T)$ has no $p$-torsion;
 \item
 for any indecomposable factor in the root system of $(G,T)$, $p$ is strictly larger than the corresponding value in Figure~\ref{fig:bounds}.
\end{enumerate}
In particular, the first assumption ensures that $G^\vee_\bk$ has simply connected derived subgroup, and the second one that its scheme-theoretic center is smooth (see~\S\ref{ss:representations-Iadj}). Finally, the third condition implies that $p$ is good for $G^\vee_\bk$.

We will apply the constructions of~\S\S\ref{ss:Waff}--\ref{ss:completions} to the latter group (with the Borel subgroup $B^\vee_\bk$ and maximal torus $T^\vee_\bk$). In particular, we fix a Steinberg section $\Sigma \subset G^\vee_\bk$, and consider the associated category $\SRep(\bbI_\Sigma^\wedge)$, see~\S\ref{ss:completions}. Note that the affine Weyl groups considered in~\S\ref{ss:Waff}, and their function $\ell$, identify with the groups $W$ and $\Waff$ of~\S\ref{ss:sheaves-Fl} and their function $\ell$. The corresponding subset of simple reflections will now be denoted $\Saff$; it coincides with the subset of $\Waff$ consisting of elements of length $1$. The subgroup of $W$ consisting of elements of length $0$ will be denoted $\Omega$. (This subgroup identifies with the subgroup $\mathbf{\Omega}$ of~\S\ref{ss:Waff} in this case.)

One of the main results of~\cite{reg-quotient-pt2} is a description of the monoidal category $(\sfT^\wedge_{\Iwu,\Iwu}, \hatstar)$ in ``Soergel bimodules'' terms. Namely, by~\cite[Theorem~11.2]{reg-quotient-pt2} there exists an equivalence of additive monoidal categories
\[
 \Phi : (\sfT^\wedge_{\Iwu,\Iwu}, \hatstar) \simto (\SRep(\bbI_\Sigma^\wedge), \oast)
\]
which satisfies
\begin{equation}
\label{eqn:Phi-Tilt}
 \Phi(\scT^\wedge_s) \cong \scB^\wedge_s
\end{equation}
for any $s \in \Saff$, and
\[
 \Phi(\scT^\wedge_w) \cong \scM^\wedge_w
\]
for any $w \in \Omega$. (This functor also satisfies some kind of compatibility with the equivalence~\eqref{eqn:Satake-equiv}, but we will not use this here.)

\begin{rmk}
\label{rmk:assumptions-qlog}
 Later we will also want to apply Lemma~\ref{lem:Abe-add-mul} in this setting. Our assumptions ensure that $p$ is very good for $G^\vee_\bk$ (so that the lemma applies, see Remarks~\ref{rmk:assumptions-Hecke-0} and~\ref{rmk:assumptions-Hecke}), except if $G$ has a component of type $\mathbf{A}_n$ and $p$ divides $n+1$. This lemma also applies for any $p$ if $G=\mathrm{GL}_n$ (see the same remarks), so that the latter case is also somewhat covered.
\end{rmk}

\subsection{Rigidification and Bott--Samelson objects}
\label{ss:rigid}

Below we will need the following construction from~\cite[Remark~11.9]{reg-quotient-pt2}.

As explained in~\cite[\S 6.6]{reg-quotient-pt2}, in case $s \in \Saff \cap \Wf$, the object $\scT_s^\wedge$ has a canonical representative, denoted $\Xi^\wedge_{s,!}$ in \emph{loc.}~\emph{cit.}, and that will be denoted $\scT_s^{\wedge,\can}$ here. For this object, the isomorphism~\eqref{eqn:Phi-Tilt} is canonical.
Let us fix some
representatives $\scT^{\wedge,\can}_{s}$ ($s \in \Saff \smallsetminus \Wf$) and $\scT^{\wedge,\can}_\omega$ ($\omega \in \Omega$) in $\sfT^\wedge_{\Iwu,\Iwu}$ for the objects $\scT^{\wedge}_{s}$ and $\scT^{\wedge}_{\omega}$, together with some identifications
\[
\Phi(\scT^{\wedge,\can}_{s}) = \scB^\wedge_s, \qquad \Phi(\scT^{\wedge,\can}_{\omega}) = \scM^\wedge_\omega
\]
($s \in \Saff \smallsetminus \Wf$, $\omega \in \Omega$.) 
Using the isomorphisms~\eqref{eqn:conjugation-Bs} and monoidality of $\Phi$, we deduce canonical isomorphisms
\begin{equation}
\label{eqn:conjugation-Xi-can}
\scT^{\wedge,\can}_{\omega} \hatstar \scT^{\wedge,\can}_{s} \hatstar \scT^{\wedge,\can}_{\omega^{-1}} \cong \scT^{\wedge,\can}_{\omega s \omega^{-1}}
\end{equation}
for any $s \in \Saff$ and $\omega \in \Omega$.

One can then define the category $\mathsf{T}^{\wedge,\mathrm{BS}}_{\Iwu,\Iwu}$ with
\begin{itemize}
\item
objects the collections $(\omega, s_1, \dots, s_i)$ with $\omega \in \Omega$ and $s_1, \dots, s_i \in \Saff$;
\item
morphisms from $(\omega, s_1, \dots, s_i)$ to $(\omega', s'_1, \dots, s'_j)$ given by
\[
\Hom_{\mathsf{T}^\wedge_{\Iwu,\Iwu}}(\scT^{\wedge,\can}_{\omega} \hatstar \scT^{\wedge,\can}_{s_1} \hatstar \cdots \hatstar \scT^{\wedge,\can}_{s_i}, \scT^{\wedge,\can}_{\omega'} \hatstar \scT^{\wedge,\can}_{s'_1} \hatstar \cdots \hatstar \scT^{\wedge,\can}_{s'_j} \bigr).
\]
\end{itemize}
(In fact, using support considerations one sees that this space vanishes unless $\omega=\omega'$.)

Using the isomorphisms~\eqref{eqn:conjugation-Xi-can} one can define on $\mathsf{T}^{\wedge,\mathrm{BS}}_{\Iwu,\Iwu}$ a monoidal structure, such that we have an equivalence of monoidal categories
\begin{equation}
\label{eqn:equiv-TBS-BSRep}
\sfT^{\wedge,\mathrm{BS}}_{\Iwu,\Iwu} \simto \BSRep(\bbI_\Sigma^\wedge)
\end{equation}
which is the identity on objects,
where $\BSRep(\bbI_\Sigma^\wedge)$ is as in~\S\ref{ss:completions} (for $\bG=G^\vee_\bk$).
We also have a canonical fully faithful monoidal functor
\begin{equation*}
\mathsf{T}^{\wedge,\mathrm{BS}}_{\Iwu,\Iwu} \to \mathsf{T}^\wedge_{\Iwu,\Iwu}
\end{equation*}
sending $(\omega, s_1, \dots, s_i)$ to $\scT^{\wedge,\can}_{\omega} \hatstar \scT^{\wedge,\can}_{s_1} \hatstar \cdots \hatstar \scT^{\wedge,\can}_{s_i}$, and $\mathsf{T}^{\wedge}_{\Iwu,\Iwu}$ identifies with the karoubian closure of the additive hull of $\mathsf{T}^{\wedge,\mathrm{BS}}_{\Iwu,\Iwu}$.

\section{Parity complexes and mixed perverse sheaves}

In~\S\S\ref{ss:parity-comp}--\ref{ss:mixed-perv} we allow $G$ to be any connected reductive algebraic group over $\F$, and allow any choice for $p$ (as long as it is invertible in $\F$).

\subsection{Parity complexes}
\label{ss:parity-comp}


We consider the category $\sfD_{\Iw,\Iw}$ from~\S\ref{ss:sheaves-Fl}, with its convolution product $\star_\Iw$, and the notion of \emph{parity complexes} in this category from~\cite{jmw}. The full subcategory of $\sfD_{\Iw,\Iw}$ whose objects are the parity complexes will be denoted $\Par_{\Iw,\Iw}$; it is stable under the bifunctor $\star_{\Iw}$. This subcategory has a more ``concrete'' description as follows. For any $s \in \Saff$, the simple perverse sheaf $\IC_s$ is just the constant sheaf on the closure of the $\Iw$-orbit labelled by $s$ (a smooth variety, isomorphic to $\mathbb{P}^1$), shifted by $1$; in particular it is a parity complex. On the other hand, if $\omega \in \Omega$ then the orbit associated with $\omega$ is just a point; in particular, $\IC_\omega$ is the skyscraper sheaf at that point, and is also a parity complex. We will denote by $\ParBS_{\Iw,\Iw}$ the category with:
\begin{itemize}
\item
objects the collections $(s_1, \dots, s_i, \omega,n)$ with $s_1, \dots, s_i \in \Saff$, $\omega \in \Omega$ and $n \in \Z$;
\item
morphisms from $(s_1, \dots, s_i,\omega,n)$ to $(s'_1, \dots, s'_j,\omega', n')$ given by
\[
\Hom_{\Par_{\Iw,\Iw}}(\IC_{s_1} \star_{\Iw} \cdots \star_{\Iw} \IC_{s_i} \star_{\Iw} \IC_\omega [n], \IC_{s_1'} \star_{\Iw} \cdots \star_{\Iw} \IC_{s_j'} \star_{\Iw} \IC_{\omega'}[n']).
\]
\end{itemize}
(In fact, using support considerations one sees that the morphism space above vanishes unless $\omega=\omega'$.)

By definition there exists a canonical fully faithful functor
\begin{equation}
\label{eqn:functor-ParBS-Par}
\ParBS_{\Iw,\Iw} \to \sfD_{\Iw,\Iw}
\end{equation}
which takes values in $\Par_{\Iw,\Iw}$.
It is easily seen that for any $\omega \in \Omega$ and $s \in \Saff$ there is a canonical isomorphism
\[
\IC_\omega \star_{\Iw} \IC_{s} \star_{\Iw} \IC_{\omega^{-1}} \cong \IC_{\omega s \omega^{-1}}.
\]
Using this property one obtains that there exists a natural convolution product (still denoted $\star_{\Iw}$) on $\ParBS_{\Iw,\Iw}$ which is defined on objects by
\begin{multline*}
 (s_1, \dots, s_i,\omega,n) \star_{\Iw} (s'_1, \dots, s'_j,\omega',n') = \\
 (s_1, \dots, s_i,\omega s'_1 \omega^{-1}, \dots, \omega s'_j \omega^{-1},\omega\omega',n+n')
\end{multline*}
and such that~\eqref{eqn:functor-ParBS-Par} is monoidal. For any $n \in \Z$ the cohomological shift functor $[n]$ induces an autoequivalence of $\ParBS_{\Iw,\Iw}$, which will again be denoted $[n]$.

It is well known that the category $\sfD_{\Iw,\Iw}$ is Krull--Schmidt, and that an object in $\sfD_{\Iw,\Iw}$ is a parity complex if and only if it is a direct sum of direct summands of objects of $\ParBS_{\Iw,\Iw}$. 
In other words, the functor~\eqref{eqn:functor-ParBS-Par} identifies $\Par_{\Iw,\Iw}$ with the karoubian envelope of the additive hull of the category $\ParBS_{\Iw,\Iw}$.

The theory developed in~\cite{jmw} provides a classification of the indecomposable objects in $\Par_{\Iw,\Iw}$. More specifically, for any $w \in W$ there exists a unique (up to isomorphism) indecomposable object $\scE_w$ in $\Par_{\Iw,\Iw}$ which is supported on the closure of the $\Iw$-orbit labelled by $w$ and whose restriction to this orbit is $\underline{\bk}[\ell(w)]$. Then the assignment $(w,n) \mapsto \scE_w[n]$ induces a bijection between $W \times \Z$ and the set of isomorphism classes of indecomposable objects in $\Par_{\Iw,\Iw}$.

\begin{rmk}
The objects $\scE_w$ have concrete and canonical descriptions in case $\ell(w) \in \{0,1\}$ (namely, these complexes are the appropriate shifts of the constant sheaves on the closures of the corresponding orbits), but not in general.
\end{rmk}

\subsection{\texorpdfstring{$\Iwu$}{Iu}-equivariant parity complexes}
\label{ss:Iwu-parity}

We also have similar notions in the category $\sfD_{\Iwu,\Iw}$; by definition, an object $\scF$ in $\sfD_{\Iw,\Iw}$ is a parity complex if and only if $\For^{\Iw}_{\Iwu}(\scF)$ is a parity complex. If we denote by $\ParBS_{\Iwu,\Iw}$ the category with:
\begin{itemize}
\item
objects the collections $(s_1, \dots, s_i,\omega,n)$ with $s_1, \dots, s_i \in \Saff$, $\omega \in \Omega$ and $n \in \Z$;
\item
morphisms from $(s_1, \dots, s_i,\omega,n)$ to $(s'_1, \dots, s'_j,\omega',n')$ given by
\begin{multline*}
\Hom_{\sfD_{\Iwu,\Iw}} \left( \For^{\Iw}_{\Iwu} ( \IC_{s_1} \star_{\Iw} \cdots \star_{\Iw} \IC_{s_i} \star_{\Iw} \IC_\omega[n]), \right. \\
\left. \For^{\Iw}_{\Iwu}(\IC_{s_1'} \star_{\Iw} \cdots \star_{\Iw} \IC_{s_j'} \star_{\Iw} \IC_{\omega'} [n']) \right),
\end{multline*}
\end{itemize}
and by $\Par_{\Iwu,\Iw}$ the full subcategory of $\sfD_{\Iwu,\Iw}$ whose objects are the parity complexes, then $\Par_{\Iwu,\Iw}$ identifies with the karoubian envelope of the additive hull of the category $\ParBS_{\Iwu,\Iw}$.

The right action of the category $\sfD_{\Iw,\Iw}$ on $\sfD_{\Iwu,\Iw}$ (by convolution) induces a right action of $\ParBS_{\Iw,\Iw}$ on $\ParBS_{\Iwu,\Iw}$, and of $\Par_{\Iw,\Iw}$ on $\Par_{\Iwu,\Iw}$. The corresponding bifunctors will again be denoted $\star_{\Iw}$. For any $n \in \Z$ the cohomological shift functor $[n]$ induces an autoequivalence of $\ParBS_{\Iwu,\Iw}$, which will again be denoted $[n]$.

If $\sfD$ is one of the categories $\sfD_{\Iw,\Iw}$, $\sfD_{\Iwu,\Iw}$, $\Par_{\Iw,\Iw}$, $\Par_{\Iwu,\Iw}$, $\ParBS_{\Iw,\Iw}$ or $\ParBS_{\Iwu,\Iw}$ and $\scF,\scG$ are objects in $\sfD$, then we will set
\[
\Hom^\bullet_{\sfD}(\scF,\scG) = \bigoplus_{n \in \Z} \Hom_{\sfD}(\scF, \scG[n]).
\]
(Depending on the context, this space will be considered either as a graded vector space, or a plain vector space.)
We will see $\bk$ as a graded $\mathsf{H}^{\bullet}_{\Iw}(\mathrm{pt}; \bk)$-module concentrated in degree $0$, in the standard way. The following lemma states a standard property of parity complexes; see e.g.~\cite[Lemma~2.2]{mr}.

\begin{lem}
\label{lem:Hom-parity}
 For any $\scF,\scG$ in $\Par_{\Iw,\Iw}$, the functor $\For^{\Iw}_{\Iwu}$ induces an isomorphism of graded vector spaces
 \[
  \bk \otimes_{\mathsf{H}^{\bullet}_{\Iw}(\mathrm{pt}; \bk)} \Hom^\bullet_{\sfD_{\Iw,\Iw}} ( \scF,\scG ) \simto \Hom^\bullet_{\sfD_{\Iwu,\Iw}} \bigl( \For^{\Iw}_{\Iwu}(\scF),\For^{\Iw}_{\Iwu}(\scG) \bigr).
 \]
\end{lem}

Below we will use the following consequences of this lemma:
\begin{enumerate}
 \item 
 \label{it:comment-Parity-1}
 the category $\ParBS_{\Iwu,\Iw}$ identifies with the category whose objects are those of $\ParBS_{\Iw,\Iw}$, and whose morphism space from $\scF$ to $\scG$ is given by the degree-$0$ part in
 \[
  \bk \otimes_{\mathsf{H}^{\bullet}_{\Iw}(\mathrm{pt}; \bk)} \Hom^\bullet_{\ParBS_{\Iw,\Iw}} ( \scF,\scG );
 \]
 \item
 for any $w \in W$, the object $\For^{\Iw}_{\Iwu}(\scE_w)$ is indecomposable; as a consequence, the assignment $(w,n) \mapsto \For^{\Iw}_{\Iwu}(\scE_w)[n]$ induces a bijection between $W \times \Z$ and the set of isomorphism classes of indecomposable objects in $\Par_{\Iwu,\Iw}$. (See~\cite[Lemma~2.4]{mr} for details.)
\end{enumerate}

\subsection{\texorpdfstring{$p$}{p}-Kazhdan--Lusztig polynomials}
\label{ss:pKL}

One possible definition of the \emph{$p$-Kazh\-dan--Lusztig polynomials} attached to $W$ is as follows: for $y,w \in W$ we set
\[
 \ph_{y,w}(v) = \sum_{n \in \Z} \dim \mathsf{H}^{-\ell(w)-n} \bigl( \Fl_{G,y}, \scE_w{}_{|\Fl_{G,y}} \bigr) \cdot v^n.
\]
(The fact that this definition coincides with that considered e.g.~in~\cite{jw} follows from the results of~\cite[Part~III]{rw}. In general, these are \emph{Laurent polynomials} rather than polynomials in the usual sense.)

Below we will use the following standard properties of these polynomials. (For Item~\eqref{it:properties-ph-1}, see e.g.~the proof of~\cite[Proposition~4.2(4)]{jw}. For~\eqref{it:properties-ph-2}, see e.g.~\cite[Proposition~2.6]{jmw}.)

\begin{lem}
 \phantomsection
 \label{lem:properties-ph}
 \begin{enumerate}
  \item 
  \label{it:properties-ph-1}
  For any $w,y \in W$ we have $\ph_{y,w}(v) = \ph_{y^{-1},w^{-1}}(v)$.
  \item
  \label{it:properties-ph-2}
  For any $w,y \in W$ we have
  \[
   \dim \bigl( \Hom^\bullet_{\Par_{\Iwu,\Iw}} ( \For^{\Iw}_{\Iwu}(\scE_w), \For^{\Iw}_{\Iwu}(\scE_y) ) \bigr) = \sum_{z \in W} \ph_{z,w}(1) \cdot \ph_{z,y}(1).
  \]
 \end{enumerate}
\end{lem}

\subsection{Mixed perverse sheaves}
\label{ss:mixed-perv}

Following~\cite{modrap2}, we define the ``mixed derived category'' of $\Iwu$-equivariant $\bk$-sheaves on $\Fl_G$ by
\[
 \Dmix_{\Iwu,\Iw} := \Kb \Par_{\Iwu,\Iw}.
\]
This category admits a ``Tate twist'' autoequivalence $\langle 1 \rangle$ defined as $\{-1\}[1]$ where $\{-1\}$ is the autoequivalence induced by the negative cohomological shift in the category $\Par_{\Iwu,\Iw}$, while $[1]$ is the cohomological shift in the homotopy category.

The constructions of~\cite[\S2]{modrap2} endow $\Dmix_{\Iwu,\Iw}$ with a ``perverse t-structure'' whose heart is a finite-length abelian category, stable under $\langle 1 \rangle$, and which will be denoted $\Pmix_{\Iwu,\Iw}$. By~\cite[\S 3.2 and Section~4]{modrap2} the category $\Pmix_{\Iwu,\Iw}$ admits a natural structure of graded highest weight category, defined by some families of ``standard objects'' $(\Delta^\mix_w : w \in W)$ and ``costandard objects'' $(\nabla^\mix_w : w \in W)$. In particular there is a notion of tilting object in $\Pmix_{\Iwu,\Iw}$, defined as an object which admits both a filtration with subquotients of the form $\Delta^\mix_w \langle n \rangle$ ($w \in W$, $n \in \Z$) and a filtration with subquotients of the form $\nabla^\mix_w \langle n \rangle$ ($w \in W$, $n \in \Z$). In this case also, if $\scF$ is a tilting object the number of subquotients isomorphic to $\Delta^\mix_w \langle n \rangle$ in such a filtration is well defined, and denoted
\[
 (\scF : \Delta^\mix_w \langle n \rangle).
\]
By~\cite[Proposition~A.4]{modrap2},
the indecomposable tilting objects in $\Pmix_{\Iwu,\Iw}$ are parame\-trized in a natural way by $W \times \Z$. More specifically, for $w \in W$ there exists a unique indecomposable tilting object $\scT^{\mix}_w$ which satisfies
\[
 (\scT^{\mix}_w : \Delta^\mix_w \langle n \rangle)=\delta_{n,0}
\]
for any $n \in \Z$, and
\[
 (\scT^{\mix}_w : \Delta^\mix_y \langle n \rangle) \neq 0 \quad \Rightarrow \quad y \leq w.
\]
With this notation, the assignment $(w,n) \mapsto \scT^\mix_w \langle n \rangle$ induces a bijection between $W \times \Z$ and the set of isomorphism classes of indecomposable tilting objects in $\Pmix_{\Iwu,\Iw}$.

Any object in $\Par_{\Iwu,\Iw}$ can also be seen as an object in $\Dmix_{\Iwu,\Iw}$, by identifying it with a complex concentrated in degree $0$. In particular, the image of $\For^{\Iw}_{\Iwu}(\scE_w)$ will be denoted $\scE_w^\mix$.

\subsection{Relation with the Hecke category}
\label{ss:parity-Hecke}

In this subsection we assume that the conditions considered in~\S\ref{ss:representations-Iadj} and in~\S\ref{ss:additive-Hecke} are satisfied by the group $\bG=G^\vee_\bk$.
Recall the category $\BSK_\add$ constructed in~\S\ref{ss:additive-Hecke}, and the ``right'' variant of this category constructed as in Remark~\ref{rmk:switch-abe-BS}, which we will denote $\BSK_{\mathrm{r},\add}$.
It is a standard fact that we have identifications
\[
 \mathsf{H}^{\bullet}_{\Iw}(\mathrm{pt}; \bk) = \mathsf{H}^{\bullet}_{T}(\mathrm{pt}; \bk) = \mathrm{S}(\bk \otimes_{\Z} \rmX^*(T))
\]
where $\mathrm{S}$ denotes the symmetric algebra (over $\bk$) and the right-hand side is seen as a graded ring with $\bk \otimes_{\Z} \rmX^*(T)$ in degree $2$. Moreover $\bk \otimes_{\Z} \rmX^*(T)$ identifies canonically with the Lie algebra $\bt$ of $T^\vee_\bk$; in this way, $\mathsf{H}^{\bullet}_{\Iw}(\mathrm{pt}; \bk)$ identifies with the graded algebra $\scO(\bt^*)$ considered in~\S\ref{ss:additive-Hecke}. The category $\BSK_{\mathrm{r},\add}$ is related to $\ParBS_{\Iw,\Iw}$ as follows.

\begin{thm}
\label{thm:Hecke-Parity}
There exists a canonical equivalence of monoidal categories
\[
\BSK_{\mathrm{r},\add} \cong \ParBS_{\Iw,\Iw}
\]
which intertwines the shift functors $(1)$ and $[1]$, and is the identity on objects.
\end{thm}

\begin{proof}
This theorem is essentially obtained as the combination of~\cite[Theorem~10.7.1]{rw} and the main result of~\cite{abe}. More precisely, these references provide a canonical equivalence of monoidal categories with the expected properties between the full subcategories in $\BSK_{\mathrm{r},\add}$ and $\ParBS_{\Iw,\Iw}$ whose objects are of the form $(s_1, \dots, s_i,e)$ with $s_1, \dots, s_i \in \Saff$. However, it is easily seen that this equivalence intertwines, for any $\omega \in \Omega$, the equivalences given by
\[
M \mapsto F_\omega \star M \star F_{\omega^{-1}} \quad \text{and} \quad \scF \mapsto \IC_\omega \star_{\Iw} \scF \star_{\Iw} \IC_{\omega^{-1}}.
\]
Using this property one sees that the equivalence above can be ``extended'' to the equivalence of the theorem.
\end{proof}

\begin{rmk}
A different (and more direct) proof of Theorem~\ref{thm:Hecke-Parity} can be obtained following the constructions in~\cite[\S 3]{abe2}. We will not pursue this here.
\end{rmk}

\section{Applications}
\label{sec:applications}

Recall the assumptions we have imposed in~\S\ref{ss:relation-Hecke}. From now on, in addition we assume that condition~\eqref{it:bilin-form} of~\S\ref{ss:additive-Hecke} holds for the group $G^\vee_\bk$, and also that the condition in Lemma~\ref{lem:Abe-add-mul} holds for this group. (See Remark~\ref{rmk:assumptions-qlog} for comments on this assumption.)

\subsection{A degrading functor}
\label{ss:Hecke-statement}

Recall the constructions of~\S\ref{ss:rigid}.
We will denote by $\sfTBS_{\Iwu,\Iw}$ the category with
\begin{itemize}
\item
objects the collections $(\omega, s_1, \dots, s_i)$ with $\omega \in \Omega$ and $s_1, \dots, s_i \in \Saff$;
\item
morphisms from $(\omega, s_1, \dots, s_i)$ to $(\omega', s'_1, \dots, s'_j)$ given by
\[
\Hom_{\mathsf{T}_{\Iwu,\Iw}} \bigl( \pi_\dag(\scT^{\wedge,\can}_\omega \hatstar \scT^{\wedge,\can}_{s_1} \hatstar \cdots \hatstar \scT^{\wedge,\can}_{s_i}), \pi_\dag(\scT^{\wedge,\can}_{\omega'} \hatstar \scT^{\wedge,\can}_{s'_1} \hatstar \cdots \hatstar \scT^{\wedge,\can}_{s'_j}) \bigr).
\]
\end{itemize}
Then we have
a canonical fully faithful functor
\begin{equation}
\label{eqn:functor-T-TBS-Iw}
\sfTBS_{\Iwu,\Iw} \to \mathsf{T}_{\Iwu,\Iw}
\end{equation}
which 
identifies $\sfT_{\Iwu,\Iw}$ with the karoubian closure of the additive hull of $\sfTBS_{\Iwu,\Iw}$. By construction the objects in $\sfT^{\wedge,\mathrm{BS}}_{\Iwu,\Iwu}$ are the same as those of $\sfTBS_{\Iwu,\Iw}$, and, by property~\eqref{it:tilting-pidag} in~\S\ref{ss:tilt-perv}, for $x,y \in \sfT^{\wedge,\mathrm{BS}}_{\Iwu,\Iwu}$ we have a canonical isomorphism
\[
\Hom_{\sfTBS_{\Iwu,\Iw}}(x,y) = \Hom_{\sfT^{\wedge,\mathrm{BS}}_{\Iwu,\Iwu}}(x,y) \otimes_{\scO((T^\vee_\bk)^\wedge)} \bk.
\]


\begin{thm}
\label{thm:Hecke-Tilt}
 There exist a functor
 \[
  \sv : \ParBS_{\Iwu,\Iw} \to \sfTBS_{\Iwu,\Iw}
 \]
and an isomorphism $\varepsilon : \sv \circ [1] \simto \sv$
such that:
\begin{enumerate}
 \item 
 \label{it:Hecke-Tilt-1}
 for any $\scF,\scG$ in $\ParBS_{\Iwu,\Iw}$, the functor $\sv$ and the isomorphism $\varepsilon$ induce an isomorphism
 \[
  \Hom^\bullet_{\ParBS_{\Iwu,\Iw}}(\scF,\scG) \simto \Hom_{\sfTBS_{\Iwu,\Iw}}(\sv(\scF),\sv(\scG));
 \]
 \item
 \label{it:Hecke-Tilt-2}
for any $s_1, \dots, s_i \in \Saff$, $\omega \in \Omega$ and $n \in \Z$ we have
\[
\sv(s_1, \dots, s_i,\omega,n)= (\omega^{-1},s_i, \dots, s_1).
\]
\end{enumerate}
\end{thm}

\begin{proof} 
Using Theorem~\ref{thm:Hecke-Parity} and comment~\eqref{it:comment-Parity-1} after Lemma~\ref{lem:Hom-parity} one obtains a canonical equivalence between the category $\ParBS_{\Iwu,\Iw}$ and the category $\overline{\BSK}_{\mathrm{r},\add}$ defined as follows: its objects are those of $\BSK_{\mathrm{r},\add}$,
and the morphisms from $M$ to $M'$ are given by the degree-$0$ part in
\[
\bk \otimes_{\scO(\bt^*)} \Hom^\bullet_{\BSK_{\mathrm{r},\add}}(M,M').
\]
This equivalence is the identity on objects.

On the other hand, consider the category $\BSK^\wedge$.
Using the equivalences~\eqref{eqn:equiv-BSK-BSRep} and~\eqref{eqn:equiv-TBS-BSRep} together with comment~\eqref{it:tilting-pidag} in~\S\ref{ss:tilt-perv}, we obtain a canonical equivalence between $\sfTBS_{\Iwu,\Iw}$ and the category $\underline{\BSK}^\wedge$ defined as follows: its objects are those of $\BSK^\wedge$, and the morphisms from $M$ to $M'$ are given by
\[
\Hom_{\BSK^\wedge}(M,M') \otimes_{\scO((T^\vee_\bk)^\wedge)} \bk.
\]
(Here the action of $\scO((T^\vee_\bk)^\wedge)$ on $\Hom$ spaces is the natural one, induced by the second projection $(T^\vee_\bk \times_{T^\vee_\bk / \Wf} T^\vee_\bk)^{\wedge} \to (T^\vee_\bk)^\wedge$.)
Once again, this equivalence is the identity on objects.

As explained in Remark~\ref{rmk:switch-abe-BS}, we have a canonical equivalence of categories $\BSK_{\mathrm{r},\add} \simto \BSK_\add$ sending $(s_1, \dots, s_i,\omega)$ to $(\omega^{-1}, s_i, \dots, s_1)$. This equivalence induces an equivalence between $\overline{\BSK}_{\mathrm{r},\add}$ and the category $\underline{\BSK}_{\add}$ which has the same objects as $\BSK_\add$, and morphisms from $M$ to $M'$ defined as
\[
  \Hom_{\BSK_\add}(M,M') \otimes_{\scO(\bt^*)} \bk.
\]
Therefore, to conclude the proof of Theorem~\ref{thm:Hecke-Tilt} it suffices to construct a functor
\begin{equation}
\label{eqn:vBSK}
\sv_{\BSK} : \underline{\BSK}_\add \to \underline{\BSK}^\wedge
\end{equation}
sending each collection $(\omega, s_1, \dots, s_i,n)$ to $(\omega, s_1, \dots, s_i)$ and an isomorphism $\sv_{\BSK} \circ (1) \cong \sv_{\BSK}$ such that for any $M,N$ in $\underline{\BSK}_\add$ these data induce an isomorphism
\[
\bigoplus_{n \in \Z} \Hom_{\underline{\BSK}_\add}(M,N(n)) \simto \Hom_{\underline{\BSK}^\wedge}(\sv_{\BSK}(M), \sv_{\BSK}(N)).
\]
This functor is obtained from Proposition~\ref{prop:relation-Hecke-cat-add} and Lemma~\ref{lem:Abe-add-mul}.
%
\end{proof}

\begin{rmk}
Theorem~\ref{thm:Hecke-Tilt} has a variant relating the categories $\ParBS_{\Iw,\Iw}$ and $\sfT^{\wedge,\mathrm{BS}}_{\Iwu,\Iwu}$, and involving the isomorphism appearing in Lemma~\ref{lem:Abe-add-mul}. We leave it to the interested reader to formulate this statement, and modify the proof above accordingly.
\end{rmk}

\subsection{Numerical consequence}

We now discuss an application of Theorem~\ref{thm:Hecke-Tilt} to multiplicities of standard perverse sheaves in indecomposable tilting perverse sheaves. 
Recall the objects $\scT_w$ and $\scT_w^\wedge$ defined in~\S\ref{ss:tilt-perv}.

\begin{cor}
\label{cor:mult-tilt}
 For any $w,y \in W$ we have
 \[
  (\scT^\wedge_w : \Delta^\wedge_y) = (\scT_w : \Delta_y) = \ph_{y,w}(1).
 \]
\end{cor}

\begin{proof}
The first equality follows from the definitions and~\eqref{eqn:mult-tilt-wedge}.

 Passing to karoubian closures of additive hulls (see~\S\ref{ss:Iwu-parity} and~\S\ref{ss:Hecke-statement}), the functor $\sv$ of Theorem~\ref{thm:Hecke-Tilt} induces a functor
 \[
  \Par_{\Iwu,\Iw} \to \sfT_{\Iwu,\Iw}
 \]
(still denoted $\sv$) which is a ``degrading functor'' in the sense that it satisfies property~\eqref{it:Hecke-Tilt-1} of Theorem~\ref{thm:Hecke-Tilt}. By construction we have
\begin{equation}
\label{eqn:v-image-BS}
\sv \bigl( \For^{\Iw}_{\Iwu} ( \IC^{\Iw}_{s_1} \star_{\Iw} \cdots \star_{\Iw} \IC^{\Iw}_{s_i} \star_{\Iw} \IC^{\Iw}_\omega)\bigr) \cong 
\pi_\dag \bigl( \scT_{\omega^{-1}}^\wedge \hatstar \scT^{\wedge}_{s_i} \hatstar \cdots \hatstar \scT^{\wedge}_{s_1} \bigr)
\end{equation}
for any $\omega \in \Omega$ and $s_1, \dots, s_i \in \Saff$.
For any $w \in W$, the finite-dimensional graded ring
\[
 \Hom^\bullet_{\Par_{\Iwu,\Iw}}(\For^{\Iw}_{\Iwu}(\scE_w),\For^{\Iw}_{\Iwu}(\scE_w))
\]
has a local degree-$0$ part; it is therefore local as an ungraded ring, see~\cite{gg}. This observation and the ``degrading'' property of $\sv$ show that $\sv(\For^{\Iw}_{\Iwu}(\scE_w))$ is indecomposable. Once this fact it known, 
it is not difficult to deduce from~\eqref{eqn:v-image-BS} that for any $w \in W$ we have
\[
 \sv(\For^{\Iw}_{\Iwu}(\scE_w)) \cong \scT_{w^{-1}}.
\]
We deduce that for any $w,y \in W$ we have
\[
 \dim \bigl( \Hom^\bullet_{\Par_{\Iwu,\Iw}} ( \For^{\Iw}_{\Iwu}(\scE_w), \For^{\Iw}_{\Iwu}(\scE_y) ) \bigr) = \dim \bigl( \Hom_{\sfT_{\Iwu,\Iw}}(\scT_{w^{-1}}, \scT_{y^{-1}}) \bigr).
\]
Comparing Lemma~\ref{lem:properties-ph}\eqref{it:properties-ph-2} and the formula in Item~\eqref{it:dim-Hom-tilt} of~\S\ref{ss:tilt-perv}, one then deduces (by induction on $w$, and then by induction on $y$ for fixed $w$) that for any $w,y \in W$ we have
\[
  (\scT_w : \Delta_y) = \ph_{y^{-1},w^{-1}}(1).
 \]
 Finally, the second formula of the corollary follows, using Lemma~\ref{lem:properties-ph}\eqref{it:properties-ph-1}.
\end{proof}

\begin{rmk}
\label{rmk:assumptions-relax}
 Using standard arguments (as e.g.~in~\cite[\S 9.5]{reg-quotient-pt3}) one can extend the validity of Corollary~\ref{cor:mult-tilt} to any connected reductive algebraic group $G$ and field $\bk$ of characteristic $p$, assuming only that for any indecomposable factor in the root system of $(G,T)$, $p$ is strictly larger than the corresponding value in Figure~\ref{fig:bounds}.
\end{rmk}

\subsection{Koszul duality}
\label{ss:Koszul-duality}

Another application of Theorem~\ref{thm:Hecke-Tilt} is to an alternative construction of the ``modular Koszul duality'' of~\cite{amrw}, in the special case of affine flag varieties. This construction, based on the ideas of an earlier construction in the setting of ordinary flag varieties of reductive groups~\cite{modrap2}, gives more than the methods of~\cite{amrw}: it also allows to construct a ``forgetful functor'' relating the ``mixed perverse sheaves'' of~\cite{modrap2, amrw} to ordinary perverse sheaves.

As terminology and notation suggest, one wants to think of $\Dmix_{\Iwu,\Iw}$ as a ``mixed version'' of the category $\sfD_{\Iwu,\Iw}$, and in fact the results of~\cite{amrw,prinblock} show that this category has properties similar to those of the category of mixed $\overline{\mathbb{Q}}_\ell$-sheaves in the sense of Deligne (or, more precisely, a modification considered in~\cite{bgs}; see~\cite{ar-kd}).
However, from its construction we do not have a priori any formal relation between $\Dmix_{\Iwu,\Iw}$ and $\sfD_{\Iwu,\Iw}$. Point~\eqref{it:forget} of the following theorem exactly compensates this discrepancy.

\begin{thm}
 \phantomsection
 \label{thm:koszul-duality}
 \begin{enumerate}
  \item 
  There exists an equivalence of triangulated categories
  \[
   \kappa : \Dmix_{\Iwu,\Iw} \simto \Dmix_{\Iwu,\Iw}
  \]
which satisfies $\kappa \circ \langle 1 \rangle \cong \langle -1 \rangle [1] \circ \kappa$ and
\begin{align*}
 \kappa(\Delta^\mix_w) \cong \Delta^\mix_{w^{-1}}, &\qquad \kappa(\nabla^\mix_w) \cong \nabla^\mix_{w^{-1}}, \\
 \kappa(\scT^\mix_w) \cong \scE^\mix_{w^{-1}}, &\qquad \kappa(\scE^\mix_w) \cong \scT^\mix_{w^{-1}}
\end{align*}
for any $w \in W$.
\item
\label{it:forget}
There exists a functor
\[
 \nu : \Dmix_{\Iwu,\Iw} \to \sfD_{\Iwu,\Iw}
\]
and an isomorphism of functors $\nu \circ \langle 1 \rangle \cong \nu$ such that for any $\scF,\scG$ in $\Dmix_{\Iwu,\Iw}$ the induced morphism
\[
 \bigoplus_{n \in \Z} \Hom_{\Dmix_{\Iwu,\Iw}}(\scF,\scG \langle n \rangle) \to \Hom_{\sfD_{\Iwu,\Iw}}(\nu(\scF),\nu(\scG))
\]
is an isomorphism. Moreover $\nu$ is t-exact for the perverse t-structures, and satisfies
\begin{align*}
 \nu(\Delta^\mix_w) \cong \Delta_w, &\qquad \nu(\nabla^\mix_w) \cong \nabla_w, \\
 \nu(\scT^\mix_w) \cong \scT_{w}, &\qquad \nu(\scE^\mix_w) \cong \scE_{w}
\end{align*}
for any $w \in W$
 \end{enumerate}
\end{thm}

\begin{proof}
 The proofs are identical to those of~\cite[Theorem~5.4 and Proposition~5.5]{modrap2}, taking as input Theorem~\ref{thm:Hecke-Tilt} instead of the main result of~\cite{modrap1}.
\end{proof}

\begin{rmk}
\label{rmk:nu-simples}
It is a standard fact that the simple objects in the category $\sfP_{\Iwu,\Iw}$ are in bijection with $W$, via the assignment $w \mapsto \IC_w$. (We omit the functor $\For^{\Iw}_{\Iwu}$ in the notation here.) A similar statement holds in the category $\sfP^\mix_{\Iwu,\Iw}$ (see~\cite[\S 3.1]{modrap2}): for any $w \in W$ the image $\IC_w^\mix$ of the unique (up to scalar) morphism $\Delta_w^\mix \to \nabla_w^\mix$ is simple, and the assignment $(w,n) \mapsto \IC^\mix_w \langle n \rangle$ induces a bijection between $W \times \Z$ and the set of isomorphism classes of simple objects in $\sfP^\mix_{\Iwu,\Iw}$. It is easily seen that for any $w \in W$ we have $\nu(\IC_w^\mix) \cong \IC_w$.
\end{rmk}


\subsection{Whittaker and parahoric versions}
\label{ss:Whit-par}


Recall that a subset $K \subset \Saff$ is called \emph{finitary} if the subgroup $W_K \subset W$ it generates is finite. (Typical examples of finitary subsets are $K=\varnothing$ and $K=\Sf$.) In this case, we will denote by $w_K$ the longest element in $W_K$. 

To a finitary subset $K \subset \Saff$ one can associate a parahoric subgroup $Q_K \subset \Loop G$ containing $\Iw$. Then we have the corresponding partial affine flag variety
\[
\Fl_{G,K} = (\Loop G / Q_K)_{\mathrm{fppf}},
\]
which is an ind-projective ind-scheme.
The natural quotient morphism
\[
\pi_K : \Fl_G \to \Fl_{G,K}
\]
is a Zariski locally trivial fibration with fibers isomorphic to the flag variety of a reductive algebraic group (namely, the quotient $M_K$ of $Q_K$ by its pro-unipotent radical). The $\Iw$-orbits on $\Fl_{G,K}$ for the natural action are in a canonical bijection with the quotient $W/W_K$.

\begin{ex}
In case $K=\varnothing$, resp.~$K=\Sf$, we have $\Fl_{G,\varnothing}=\Fl_G$, resp.~$\Fl_{G,\Sf} = \Gr_G$.
\end{ex}

Choose, for any $w \in \Wf$, a lift $\dot{w} \in \mathrm{N}_G(T)$ of $w$. Then we obtain lifts in $\Loop G$ of all elements of $W$ as follows: if $w = x \mathsf{t}(\lambda)$ with $x \in \Wf$ and $\lambda \in \rmX_*(T)$ we set $\dot{w} = \dot{x} z^\lambda$, were $z^\lambda \in \Loop T$ is the point naturally associated with $\lambda$. We also fix, for any positive root $\alpha$, a morphism $\varphi_\alpha : \mathrm{SL}_2 \to G$ which satisfies the natural conditions spelled out e.g.~in~\cite[\S 3.4]{ar-steinberg}.

We continue with a finitary subset $K \subset \Saff$ as above, and
let now $L \subset \Saff$ be another finitary subset. We set $\Iwu^L = \dot{w}_L \Iwu (\dot{w}_L)^{-1}$. Then the quotient $\Iwu^L / (\Iwu^L \cap \Iwu)$ identifies with the unipotent radical of a Borel subgroup of the reductive quotient $M_L$. Our choice of morphisms $\varphi_\alpha$ determines a morphism from this group to the additive group $\Ga$, see~\cite[\S 3.4]{ar-steinberg}, and we denote by $\psi_L : \Iwu^L \to \Ga$ the composition with the projection $\Iwu^L \to \Iwu^L / (\Iwu^L \cap \Iwu)$. Assuming that there exists a nontrivial $p$-th root of unity in $\bk$ (which we fix), we obtain an Artin--Schreier local system $\mathrm{AS}$ on $\Ga$, and we consider the category
\[
\Db_{(\Iwu^L, \psi_L^* \mathrm{AS})}(\Fl_{G,K})
\]
of $(\Iwu^L, \psi_L^* \mathrm{AS})$-equivariant $\bk$-sheaves on $\Fl_{G,K}$, and the subcategory
\[
\Perv_{(\Iwu^L, \psi_L^* \mathrm{AS})}(\Fl_{G,K})
\]
of perverse sheaves.

\begin{ex}
In case $L=\varnothing$ we have $\Iwu^\varnothing = \Iwu$, and $\Psi_\varnothing$ is the trivial morphism. 
In this case we do not need to assume that $\bk$ contains a nontrivial $p$-th root of unity, and we write $\Db_{\Iwu}(\Fl_{G,K})$, $\Perv_{\Iwu}(\Fl_{G,K})$ instead of $\Db_{(\Iwu^\varnothing, \psi_\varnothing^* \mathrm{AS})}(\Fl_{G,K})$, $\Perv_{(\Iwu^\varnothing, \psi_\varnothing^* \mathrm{AS})}(\Fl_{G,K})$.
On the other hand, if $L=\Sf$ the group $\Iwu^{\Sf}$ is the preimage under the projection $\Loop^+ G \to G$ of the unipotent radical of the Borel subgroup opposite to $B$ with respect to $T$. The morphism $\psi_{\Sf}$ is the composition of the restriction of the latter morphism with a generic additive character of the unipotent radical.
\end{ex}

The set of $\Iwu^L$-orbits in $\Fl_K$ is in a canonical bijection with the double quotient $W_L \backslash W / W_K$, but not all orbits support nonzero $(\Iwu^L, \psi_L^* \mathrm{AS})$-equivariant local systems. More specifically, denote by ${}^L W^K \subset W$ the subset of elements $w$ which satisfy $\ell(w_L w w_K) = \ell(w_L) + \ell(w) + \ell(w_K)$. (See~\cite[Lemma~2.4]{ar-steinberg} for alternative characterizations of these elements---this statement only considers the case $K=\Sf$, but the general case is similar.) These elements are minimal in their coset in $W_L \backslash W / W_K$; in particular, each double coset contains at most one element which satisfies this property. But not every double coset contains such an element, except in the special case where $K$ or $L$ is empty. With this notation, the orbit corresponding to a double coset supports a nonzero $(\Iwu^L, \psi_L^* \mathrm{AS})$-equivariant local system iff it contains an element $w$ in ${}^L W^K$; in this case, there exists a unique irreducible such local system. Taking $!$-extension, $*$-extension, and intermediate extension of this local system (shifted by the dimension of the orbit) we obtain objects
\[
{}^L \Delta_w^K, \quad {}^L \nabla_w^K, \quad {}^L \hspace{-1pt} \IC_w^K
\]
in $\Perv_{(\Iwu^L, \psi_L^* \mathrm{AS})}(\Fl_{G,K})$. Then $\Perv_{(\Iwu^L, \psi_L^* \mathrm{AS})}(\Fl_{G,K})$ is a highest weight category with weight poset ${}^L W^K$ (for the restriction of the order $\leq$ on $W$ considered in~\S\ref{ss:sheaves-Fl}), standard objects the objects $({}^L \Delta_w^K : w \in {}^L W^K)$, and costandard objects the objects $({}^L \nabla_w^K : w \in {}^L W^K)$. In particular, one can consider the notion of tilting objects in $\Perv_{(\Iwu^L, \psi_L^* \mathrm{AS})}(\Fl_{G,K})$, and we have a bijection $w \mapsto {}^L \hspace{-1pt} \scT^K_w$ between ${}^L W^K$ and the set of isomorphism classes of indecomposable tilting objects in this category. For a tilting object $\scF$ one can also consider the multiplicity $(\scF : {}^L \Delta_w^K)$ of a given object ${}^L \Delta_w^K$ in a filtration with standard subquotients.

The main result of the present subsection is the following result, which generalizes Corollary~\ref{cor:mult-tilt}. (The latter statement corresponds to the case $K=L=\varnothing$ of the present theorem.)

\begin{thm}
\label{thm:parabolic-Whit}
For any $y,w \in {}^L W^K$ we have
\[
({}^L \hspace{-1pt} \scT^K_w : {}^L \Delta_y^K) = \sum_{x \in W_K} (-1)^{\ell(x)} \cdot {}^p \hspace{-1pt} h_{yx, w_L w}(1).
\]
\end{thm}

\begin{proof}
The first step is to reduce the proof to the case $L=\varnothing$. For that, recall that we have an ``averaging'' functor
\begin{equation*}
\Db_{\Iwu}(\Fl_{G,K}) \to
\Db_{(\Iwu^L, \psi_L^* \mathrm{AS})}(\Fl_{G,K}).
\end{equation*}
This functor
has a left and a right adjoint, which are t-exact, and send standard, resp.~costandard, objects to standard, resp.~costandard, objects. (These statements are proved in~\cite[\S 3.7]{ar-model} in case $K=\Sf$. The general case is similar. Similar comments apply to~\cite[Proposition~3.12]{ar-model} which is cited below.) More explicitly, the image of ${}^L \Delta_y^K$ under any of these functors admits a filtration whose associated graded is the sum of the objects ${}^\varnothing \Delta^K_{xy}$ where $x$ runs over $W_L$, and similarly for costandard objects. By~\cite[Proposition~3.12]{ar-model} these functors send ${}^L \hspace{-1pt} \scT^K_w$ to ${}^\varnothing \hspace{-1pt} \scT^K_{w_L w}$; we deduce that for $y,w \in {}^L W^K$ we have
\[
({}^L \hspace{-1pt} \scT^K_w : {}^L \Delta_y^K) = ({}^\varnothing \hspace{-1pt} \scT^K_{w_L w} : {}^\varnothing \Delta_y^K).
\]
As announced, it therefore suffices to prove the theorem in case $L=\varnothing$.

Now we consider the (smooth) morphism $\pi_K$. 
We have a t-exact 
functor
\begin{equation}
\label{eqn:pullback-piK}
(\pi_K)^\dag := (\pi_K)^*[\dim(Q_K/\Iw)] : \Db_{\Iwu}(\Fl_{G,K}) \to \Db_{\Iwu}(\Fl_G).
\end{equation}
This functor has a right adjoint
\[
(\pi_K)_\dag := (\pi_K)_*[-\dim(Q_K/\Iw)] : \Db_{\Iwu}(\Fl_G) \to \Db_{\Iwu}(\Fl_{G,K}),
\]
and we have
\[
(\pi_K)_\dag \circ (\pi_K)^\dag \cong \bigoplus_{x \in W_K} \id[-2\ell(x)].
\]
(In fact, the functor $(\pi_K)^\dag$ is given by convolution on the right with the $Q_K$-equivariant complex $\underline{\bk}_{Q_K / \Iw}[\dim(Q_K/\Iw)]$ on $\Fl_G$, and the functor $(\pi_K)_\dag$ is given by right convolution with the $\Iw$-equivariant complex $\delta_K[-\dim(Q_K/\Iw)]$, where $\delta_K$ is the skyscraper sheaf at the base point of $\Fl_{G,K}$. The composition $(\pi_K)_\dag \circ (\pi_K)^\dag$ is therefore convolution on the right with the $Q_K$-equivariant convolution of these complexes, which is the tensor product of $\bigoplus_n \mathsf{H}^n(Q_K/I;\bk)[-n]$ with the skyscraper sheaf at the base point of $\Fl_G$.) Since ${}^{\mathrm{p}} \hspace{-1pt} \mathscr{H}^0 \circ (\pi_K)_\dag \circ (\pi_K)^\dag \cong \id$, the functor~\eqref{eqn:pullback-piK} is fully faithful on perverse sheaves (which, of course, follows also from general results on perverse sheaves), and since ${}^{\mathrm{p}} \hspace{-1pt} \mathscr{H}^1 \circ (\pi_K)_\dag \circ (\pi_K)^\dag=0$ its essential image is stable under extensions. It is also a standard fact that this functor sends simple objects to simple objects; this essential image therefore coincides with the Serre subcategory generated by the objects $\IC_w$ where $w$ is maximal in $wW_K$.

Similarly, we can consider the mixed derived category $\Dmix_{\Iwu}(\Fl_{G,K})$ of $\Iwu$-equiva\-riant sheaves on $\Fl_{G,K}$ (defined as in the case $K=\varnothing$ in~\S\ref{ss:mixed-perv}). This category has a natural ``perverse'' t-structure whose heart is denoted $\Perv^\mix_{\Iwu}(\Fl_{G,K})$. The functor $(\pi_K)^\dag$ sends parity complexes to parity complexes, hence induces a functor
\[
\Dmix_{\Iwu}(\Fl_{G,K}) \to \Dmix_{\Iwu}(\Fl_{G}).
\]
The same comments as above show that this functor is t-exact, and that its restriction to perverse sheaves identifies $\Perv^\mix_{\Iwu}(\Fl_{G,K})$ with the Serre subcategory of $\Perv^\mix_{\Iwu}(\Fl_{G})$ generated by the simple objects $\IC^\mix_w \langle n \rangle$ for $w \in W$ maximal in $wW_K$ and $n \in \Z$.

Now, consider the functor
\begin{equation*}
\Pmix_{\Iwu,\Iw} \to \Perv_{\Iwu}(\Fl_G)
\end{equation*}
obtained by restriction from the functor $\nu$ of Theorem~\ref{thm:koszul-duality}\eqref{it:forget}. The comments above and Remark~\ref{rmk:nu-simples} show that this functor restricts to a functor
\[
\nu_K : \Perv^\mix_{\Iwu}(\Fl_{G,K}) \to \Perv_{\Iwu}(\Fl_{G,K}).
\]
There are standard, costandard, and tilting objects in the category $\Perv^\mix_{\Iwu}(\Fl_{G,K})$, and one can show that $\nu_K$ sends standard, costandard, tilting objects to standard, costandard, tilting objects respectively, and indecomposable objects to indecomposable objects. In particular, for any $w \in {}^\varnothing W^K$ the object ${}^\varnothing \hspace{-1pt} \scT^K_w$ is the image of the indecomposable tilting object in $\Perv^\mix_{\Iwu}(\Fl_{G,K})$ labelled by $w$. 

Now the multiplicities of standard objects in indecomposable tilting modules in $\Perv^\mix_{\Iwu}(\Fl_{G,K})$ can be obtained by copying in our present setting the constructions of~\cite[\S 6]{amrw}. (See in particular~\cite[Corollary~7.5]{amrw} for similar results.) The formula obtained in this way is exactly that of the theorem.
\end{proof}

\begin{rmk}
\begin{enumerate}
\item
The proof of Theorem~\ref{thm:parabolic-Whit} shows that, in fact, for $y,w \in {}^L W^K$, for any $z \in W_L$ we have
\[
({}^L \hspace{-1pt} \scT^K_w : {}^L \Delta_y^K) = \sum_{x \in W_K} (-1)^{\ell(x)} \cdot {}^p \hspace{-1pt} h_{zyx, w_L w}(1).
\]
\item
Given a Coxeter group and a finite parabolic subgroup, there are two families of ``parabolic'' Kazhdan--Lusztig polynomials: the $(+)$-parabolic ones, whose definition involves the ``trivial'' module for the Hecke algebra of the parabolic subgroup, and the $(-)$-parabolic ones, whose definition involves the ``sign'' module for this Hecke algebra. One can also combine these constructions (considering two finite parabolic subgroups, one acting on the left and the other one acting on the right), and there are analogues of these constructions for $p$-Kazhdan--Lusztig polynomials. The polynomials that appear in Theorem~\ref{thm:parabolic-Whit} are $(+)$-parabolic on the left for $W_L$, and $(-)$-parabolic on the right for $W_K$.
\item
There is one case of Theorem~\ref{thm:parabolic-Whit} which is particularly relevant for Representation Theory, namely when $K=\Sf$ and $L=\varnothing$. In this case, by the Finkelberg--Mirkovi{\'c} conjecture proved in~\cite{reg-quotient-pt3}, if $p$ satisfies appropriate conditions the category $\Perv_{\Iwu}(\Gr_G)$ is equivalent to the extended principal block of the category of representations of the reductive algebraic group over $\bk$ whose Frobenius twist is $G^\vee_\bk$. Under this equivalence, tilting perverse sheaves correspond to tilting representations, and the formula in Theorem~\ref{thm:parabolic-Whit} corresponds to the character formula conjectured with Williamson and first proved in~\cite{amrw}.
\item
The same arguments as in~\cite[Proposition~2.4.1]{yun} show that the functor
\[
(\pi_K)_* : \Db_{(\Iwu^L, \psi_L^* \mathrm{AS})}(\Fl_{G}) \to \Db_{(\Iwu^L, \psi_L^* \mathrm{AS})}(\Fl_{G,K})
\]
sends tilting perverse sheaves to tilting perverse sheaves. One can also easily show that if $w \in {}^L W^\varnothing \smallsetminus {}^L W^K$ we have $(\pi_K)_*({}^L \hspace{-1pt} \scT^\varnothing_w)=0$. (See~\cite[Lemma~6.3(2)]{amrw} for a similar statement for mixed perverse sheaves; the same arguments apply here.) One can deduce from Theorem~\ref{thm:parabolic-Whit} that for any $w \in {}^L W^K$ we have
\begin{equation}
\label{eqn:pushforward-tilting}
(\pi_K)_*({}^L \hspace{-1pt} \scT^\varnothing_w) \cong {}^L \hspace{-1pt} \scT^K_w.
\end{equation}
In the setting of $\ell$-adic sheaves, this statement follows from~\cite[Proposition~3.4.1]{yun} (see also~\cite[\S 5]{yun}); the proof in this case uses considerations of weights of Frobenius. We do not know any more direct proof of~\eqref{eqn:pushforward-tilting}.
\end{enumerate}
\end{rmk}

\end{document}